\documentclass[reqno,a4paper,11 pt]{amsart}
\usepackage{amsmath}
  \usepackage{paralist}
  \usepackage{graphics}
  
   \usepackage[colorlinks=true]{hyperref}
\hypersetup{urlcolor=blue, citecolor=red}

\newtheorem{theorem}{Theorem}[section]

\newtheorem{lemma}[theorem]{Lemma}
\newtheorem{proposition}{Proposition}

\theoremstyle{definition}

\newtheorem{remark}{Remark}

\usepackage{mathtools}
\usepackage{amssymb}
\usepackage{braket}
\usepackage{enumerate}
\usepackage{verbatim}
\usepackage[pagewise]{lineno}
\newtheorem{hypothesis}[theorem]{Hypothesis}

\newenvironment{system}
{\left\lbrace\begin{array}{@{}l@{}}}
{\end{array}\right.}
\DeclarePairedDelimiter{\abs}{\lvert}{\rvert}

\def\ds{\begin{displaystyle}}
\def\eds{\end{displaystyle}}
\def\dis{\displaystyle }
\def\<{\langle }
\def\>{\rangle }

\newcommand{\R}{\mathbb{R}}
\newcommand{\E}{\mathbb{E}}
\newcommand{\mP}{\mathbb{P}}
\newcommand{\Fcal}{\mathcal{F}}      
\newcommand{\Lcal}{\mathcal{L}}

\newcommand{\calb}{\mathcal{B}}

\def\R{\mathbb R}     

\def\E{\mathbb E}
\def\P{\mathbb P}
\def\Q{\mathbb Q}

\title[On the stochastic maximum principle with delay]
      {Stochastic maximum principle for problems with delay with dependence on the past through general measures}

\author[Giuseppina Guatteri and Federica Masiero]{}


 \email{giuseppina.guatteri@polimi.it}
 \email{federica.masiero@unimib.it}

\begin{document}
\maketitle

\centerline{\scshape Giuseppina Guatteri}
\medskip
{\footnotesize
 \centerline{Dipartimento di Matematica, Politecnico di Milano}
   \centerline{  via Bonardi 9, 20133 Milano, Italia}}

\medskip

\centerline{\scshape Federica Masiero}
\medskip
{\footnotesize
 \centerline{ Dipartimento di Matematica e Applicazioni, Universit\`a di Milano-Bicocca.}
   \centerline{via Cozzi 55, 20125 Milano, Italia}}
%

\begin{abstract}
We  prove a stochastic maximum principle for a control problem where the state equation is delayed both in the state and in the control, and both the running and the final cost functionals may depend on the past trajectories. The adjoint equation turns out to be a new form of linear anticipated backward stochastic differential equations (ABSDEs in the following), and we prove a direct formula to solve these equations.
\end{abstract}

\maketitle

\section{Introduction}

In this paper we consider a controlled state equation for the process $x$ with delay both in the state and in the control $u$, namely $x$ is the solution to the following stochastic delay controlled equation in $\R^n$ driven by an $m$-dimensional Brownian motion $W$:
\begin{equation}\label{eq:state-intro}
\begin{system}
dx(t)=\textcolor{blue}{f}(t,x_t,u_t)dt +\textcolor{blue}{g}(t,x_t,u_t)dW(t), \\
x(\theta)= x_0(\theta), \quad u(\theta)=\eta(\theta),\qquad \theta \in [-d, 0].
\end{system}
\end{equation}
Here and throughout the paper we use the notation  $x_t(\theta) = x(t+\theta),\, u_t(\theta) = u(t+\theta),$ with $\theta \in [-d,0]$ to denote the past trajectory of $x$ and $u$ from $t-d$ up to time $t$.  We consider admissible controls $u$, that are progressively measurable and square integrable processes taking values in a convex set  $U\subset \R^k$: in this case the stochastic maximum principle can be formulated in terms of the first order adjoint equation.
\newline We are able to allow a quite general dependence on the past trajectories $x_t$ and $u_t$ of the state and of the control, namely the drift  and diffusion can be written as 
\begin{align}\label{barf-barg-intro}
 f(t,x,u)=\bar f (t, \int_{-d}^0 x(\theta )\mu_1(d\theta),  \int_{-d}^0 u(\theta )\mu_3(d\theta)), \\ \nonumber
 \quad g(t,x,u)=\bar g (t, \int_{-d}^0 x(\theta )\mu_2(d\theta),  \int_{-d}^0 u(\theta )\mu_4(d\theta)),
 \end{align}
 where $\mu_1,\mu_2,\mu_3,\mu_4$ are finite regular measures on $[-d,0]$, and $\bar f(t,\cdot,\cdot)$ and  $\bar g(t,\cdot,\cdot)$ are Lipschitz continuous and differentiable.
Associated to equation (\ref{eq:state-intro}) we consider the cost functional 
\begin{equation}\label{costo:fin-intro}
J(u(\cdot)) = \E \int_0^T l(t,x_t, u_t) dt + \E\, h(x_T) 
\end{equation}
that we have to minimize over all admissible controls. Also in the running cost $l$ and in the final cost $h$ we can allow a general dependence on the past trajectories $x_t$ and $x_T$: there exist finite regular measures $\mu_5,\,\mu_6,\,\mu$ such that the current cost and the final cost can be written as follows
\begin{align}\label{barl-barh-intro}
l(t,x, u)=\bar l (t,\int_{-d}^0 x(\theta )\mu_5(d\theta),\int_{-d}^0u(\theta )\mu_6(d\theta)),\quad \\ \nonumber
 h(x)=\bar h (\int_{-d}^0 x(\theta )\mu(d\theta)), \; \forall \,x \in \, C([-d,0], \R^n),
\end{align}
Such kind of dependence is rather general even though some relevant cases cannot be covered, like the dependence on the supremum of the history of the path,  for example $ h(x)=h( \sup_{\theta \in [-d,0]}x(\theta))$.

We choose to attach our problem by means of the stochastic maximum principle since the dynamic programming approach stochastic optimal control problems governed by delay equations with delay in the control are usually harder to study than the ones when the delay appears only in the state.  The main difficulty is that the associated Hamilton Jacobi Bellman equation is an infinite dimensional second order semilinear PDE, which is not trivial to solve, see e.g. \cite{GozMar,GozMarSav,GozMas}. Indeed the delay in the control cannot be directly treated by means of the dynamic programming principle, and in order to remove such delay in the control the problem must be turned into an infinite dimensional stochastic control problem that, unlike the infinite dimensional formulation of problem with delay in the state, does not satisfy the so called structure condition according to which the control affect the system as a perturbation of the noise, and moreover in many situations, including the case of pointwise delay, the control operator is unbounded. More general cases can be treated by applying the so called randomization method, see e.g. \cite{BanCosFuhPha}: with this approach it is possible to characterize the value function but no conditions on the optimal control can be given.

\noindent On the contrary, studying a stochastic optimal control problem by means of the stochastic maximum principle allows to get conditions on the optimal control. 

When studying our control problem  by means of the stochastic maximum principle the adjoint equation turns out to be the following ABSDE for the pair of processes $(p,q)\in
\Lcal^2_{\mathcal{F}}(\Omega\times [0,T], \R^{n})\times 
\Lcal^2_{\mathcal{F}}(\Omega\times [0,T], \R^{n\times  m})$, 
 \begin{equation}\label{ABSDEtrascinata-aggiunta-intro}
  \left\lbrace\begin{array}{l}
p(t)=\dis\int_t^T \E^{\Fcal_s}\dis\int_{-d}^0
\bar l_x\left(s-\theta, x(s-\theta),u(s-\theta)\right)
\mu_5(d\theta)\,ds\\
\qquad+\dis\int_t^T\E^{\Fcal_s}\dis\int_{-d}^0 p(s-\theta)
\bar f_x\left(s-\theta, x(s-\theta),u(s-\theta)\right)\mu_1(d\theta)\, ds\\
\qquad+\dis\int_t^T
\E^{\Fcal_s}\dis\int_{-d}^0q(s-\theta)\bar g_x\left(s-\theta, x(s-\theta),u(s-\theta)\right)
\mu_2(d\theta)\, ds \\
\qquad+\dis\int_t^Tq(s)dW_s+\dis\int_{t\vee(T-d)}^T\textcolor{blue}{\E^{\Fcal_s}}\ \bar  h_x(x_T)\textcolor{blue}{\mu^T(ds)},\\
\quad\qquad p(T-\theta)=0,\quad q(T-\theta)=0, \quad \forall \,\theta \in [-d,0 ),
 \end{array}
 \right.
 \end{equation}
 \textcolor{blue}{where by $\mu^T$ we have denote the measure obtained by translating $\mu$, namely for any $A\in\calb(\R)$, we have $ \mu^T(A)=\mu(A-T)$ where $A-T:=\left\lbrace x-T:x\in A \right\rbrace$}.
\newline Notice that equation \eqref{ABSDEtrascinata-aggiunta-intro} does not make sense in differential form: the term 
 $$\frac{d}{dt}\dis\int_{t\vee(T-d)}^T\textcolor{blue}{\E^{\Fcal_s}}\bar  h_x(x_T)\textcolor{blue}{\mu^T(ds)}$$ 
  is well defined when $\mu$ is absolutely continuous with respect to the Lebesgue measure. In order to be able to work with differentials, we will consider an ABSDE where $\mu$ is approximated by a sequence of finite regular measures $(\mu^n)_{n\geq 1}$ absolutely continuous with respect to the Lebesgue measure on $[T-d,T]$, so that the differential
\[
\frac{d}{dt}\int_{t\vee(T-d)}^T\textcolor{blue}{\E^{\Fcal_s}}\bar  h_x(x_T)\textcolor{blue}{\mu^{n,T}(ds)}
\] 
 makes sense (\textcolor{blue}{here $ \mu^{n,T}(A)=\mu^n(A-T)$}). In this way, for the approximating ABSDE the differential form makes sense. 
\newline The ABSDE \eqref{ABSDEtrascinata-aggiunta-intro} is a new type of linear ABSDEs, already considered in \cite{GMO}: in the present paper for  the solution of \eqref{ABSDEtrascinata-aggiunta-intro} we are able to give a representation which is the analogous of the one for linear BSDEs.

With these tools in hands, we are able to state necessary conditions for the optimality in terms of the pair of processes $(p,q)$: let $(\bar x, \bar u)$ be an optimal pair and  let $u^\rho= \bar u+\rho \bar v$, where $\bar v$ is another admissible control, then, $\mathbb{P}$- a.s. and for a.e. $t \in [0,T]$:
\begin{multline}\label{eq:opt-contr-intro}
\E^{\mathcal{F}_t}\!\!\!\int_{-d}^0\!\!\!\left(\bar{f}(t-\theta, \bar{x}(t-\theta), u^\rho(t-\theta))\!-\!\!\bar{f}(t-\theta,\bar x(t-\theta),\bar u(t-\theta))\right)p(t-\theta)\mu_3(d\theta) \\
+\E^{\mathcal{F}_t}\int_{-d}^0\textcolor{blue}{\left(l(t-\theta, \bar{x}(t-\theta), u^\rho(t-\theta))- l(t-\theta, \bar x(t-\theta),\bar u(t-\theta))\right)}\,\mu_6(d\theta) +\\
 \E^{\mathcal{F}_t} \int_{-d}^0\!\!\left(\bar{g}(t-\theta, \bar x{(t-\theta)}, u^\rho(t-\theta))\!-\!\bar{g}(t-\theta,\bar x(t-\theta),\bar u(t-\theta))\right)q(t-\theta)\,\mu_4(d\theta) \geq 0.
\end{multline}
Notice that this formula can be rewritten in a differential way under stronger assumptions on the coefficients, see Section \ref{Sec:contr}, formula \eqref{max-princ:fin:condiz-ham-diffle}.

The results achieved by means of the stochastic maximum principle can be applied to a stochastic optimal control problem arising in advertisement models with delay and to an optimal portfolio problem with execution delay, we refer to Sections \ref{sec:opt-adv} and \ref{sec:opt-port} for details.

After the introduction of anticipated backward stochastic differential equations (ABSDEs) in the paper \cite{PengYang}, the stochastic maximum principle for delay equations has been widely studied in the literature.
We mention, among others, \cite{ChenWuAutomatica2010}, where a problem with pointwise delay in the state and in the control is studied, \cite{oksendal2011optimal}, where a controlled state equation driven by a Brownian motion and by a Poisson random measure is taken into account, and the delay affects the system by means of terms with a more restrictive structure that the one considered in \eqref{barf-barg-intro}, indeed, the measures $\mu_j,\, j=1,...,4$ all reduce to the same measure, absolutely continuous with respect to the Lebesgue measure, and with exponential density.

In the present paper we study the stochastic maximum principle for stochastic control problems where the state equation may present delay in the state and in the control and where in the associated cost functional we allow dependence on the past trajectory also in the final cost.  Following the standard steps in the variational approach for control problems, we formulate the maximum principle by means of an adjoint equation. {The novelty is that the adjoint equation} turns out to be an ABSDE of a more general form than the ones introduced in \cite{PengYang} and generalized in \cite{ElliottYang}. This is due to the fact that we allow dependence on the past trajectory also in the final cost. It turns out that the adjoint equation is not regular enough  to  be an It\^o process and so to prove the stochastic maximum principle we must introduce some suitable regularized  approximating problems.
\newline The dependence on the past trajectory in the final cost  has been studied also in \cite{GMO}, for an infinite dimensional evolution equation with delay only in the state and no control dependent noise. The adjoint ABSDE considered in \cite{GMO} is similar to the one we handle here; in the present paper the ABSDE is solved directly by an extension of the formulas for linear BSDEs.  

Concerning the recent literature based on ABSDEs, we are able to consider more general dependence on the past trajectory and moreover we can study the case when the final cost depends on the past trajectory of the state. As far as we know, such a general case is studied only in \cite{hu1996maximum}, with a direct functional analytic method, and the authors do not take into account the delay in the control.

The paper is organized as follows: in Section \ref{sec:newABSDE} we study the new form of linear ABSDEs, in Section 3 we present the control problem and after studying the variation of the state with respect to the variation of the control we formulate and prove the stochastic maximum principle, in Sections \ref{sec:opt-adv} and \ref{sec:opt-port} the results are applied respectively to a stochastic dynamic model in marketing for problems of optimal advertising and to an optimal portfolio problem with execution delay.

\subsection{Notations}

Let $(\Omega, \Fcal, \mP)$ be a complete probability space, $ W(t)$ an $m$-dimensional Brownian motion and let $(\mathcal{F}_t)_{t\geq 0}$ be the natural filtration associated to $W$, augmented in the usual way with the family of $\mP$-null sets of $\mathcal{F}$.
\newline For any $p\in [1,\infty]$ and $T>0$ we define
\begin{itemize}
\item $\Lcal^p_{\mathcal{F}}(\Omega\times [0,T]; \R^k)$, the set of all $(\mathcal{F}_t)$-progressive processes with values in $\R^k$ such that the norm
$$
\begin{array}{lll}||Y ||^p_{\Lcal_{\mathcal{F}}^p(\Omega\times[0,T];\R^k)} = \displaystyle \left( \E \int_0^T \abs{Y(t)}_{\R^k}^p dt\right)^{1/p}
& \mbox{if}& p<\infty,
\\
||Y ||_{\Lcal_{\mathcal{F}}^\infty(\Omega\times[0,T];\R^k)} = \mathop{ess\;sup}_{\omega\in\Omega,t\in [0,T ]}
 |Y_t(\omega)|&
  \mbox{if}& p=\infty
  \end{array}
$$
is finite. Here we take the $ \mathop{ess\;sup}$ with respect to $dt \otimes d\P$.
\item $\Lcal^p_{\mathcal{F}}(\Omega;C( [0,T]; \R^k))$, the set of all  $(\mathcal{F}_t)$-progressive and continuous processes with values in $\R^k$ such that the norm
$$
\begin{array}{lll}||Y ||^p_{\Lcal^p_{\mathcal{F}}(\Omega;C([0, T];\R^k))} = \E \sup_{t\in [0,T ]} |Y_t|^p
& \mbox{if}& p<\infty,
\\
||Y ||_{\Lcal^\infty_{\mathcal{F}}(\Omega;C([0, T];\R^k))} = \mathop{ess\;sup}_{\omega\in\Omega}
\sup_{t\in [0,T ]} |Y_t(\omega)|&
  \mbox{if}& p=\infty
  \end{array}
$$
is finite.
Elements of this space are identified up to indistinguishability. We will denote the space as $\mathcal{S}^p_{\mathcal{F}}([0,T])$.

 \item $\Lcal^p_{\mathcal{F}}(\Omega;B([0, T];\R^k))$, the set of  $(\mathcal{F}_t)$-progressive measurable and a.s bounded trajectories processes with values in $\R^k$  such that the norm
$$
\begin{array}{lll}||Y ||^p_{\Lcal^p_{\mathcal{F}}(\Omega;B([0, T];\R^k))} = \E \sup_{t\in [0,T ]} |Y_t|^p
& \mbox{if}& p<\infty,
\\
||Y ||_{\Lcal^\infty_{\mathcal{F}}(\Omega;B([0, T];\R^k))} = \mathop{ess\;sup}_{\omega\in\Omega}
\sup_{t\in [0,T ]} |Y_t(\omega)|&
  \mbox{if}& p=\infty
  \end{array}
$$
is finite.
Elements of this space are identified up to indistinguishability.
 We will denote the space as $\mathcal{B}^p_{\mathcal{F}}([0,T])$
\end{itemize}

Throughout the paper given a progressive measurable process $y\in \Lcal^1_{\mathcal{F}}(\Omega\times[0,T],\R^k)$, for $0\leq s\leq t$ by $\E^{\Fcal_s} y(t)$ we denote the optional projection of $y$ into $\Fcal_s$.

\section{A new form of anticipated backward stochastic differential equations}
\label{sec:newABSDE}

In this section we study ABSDEs which have the suitable form to be the adjoint equations in problems with delay we treat in the present paper.
We will consider a stochastic differential equation of backward type, and on its coefficients we make the following assumptions.
\begin{hypothesis}\label{hyp:BSDEtrascinata}
Let
$f\in\Lcal^2_{\mathcal{F}}(\Omega\times [0,T], \R^{n})$, 
$g\in\Lcal^\infty_{\mathcal{F}}(\Omega\times [0,T], \R^{n\times  n})$,
$h\in\Lcal^\infty_{\mathcal{F}}(\Omega\times [0,T], \R^{m})$
and $\xi\in\mathcal{B}^2_{\mathcal{F}}([T-d,T])$ where $0<d\leq T$ is fixed.  Let $\mu$ be a finite  regular measure on $[T-d,T]$ and denote by $|\mu|$ its total variation, see \cite{eidelman2004functional}. 
\end{hypothesis}
Here we refer to the definition of regular measure, given in \cite{eidelman2004functional}, page 156, according to which given a closed interval $I$ a non negative measure $\mu$ defined on the $\sigma$-algebra $\mathcal{B}(I)$ of all the Borel sets of $I$ is called \textit{ regular} if $\forall\, A\in \mathcal{B}(I)$
\[
\mu(A)=\sup\{ \mu(F):F\subseteq A,\, F \text{ is closed}\}\]
\[
\mu(A)=\inf\{ \mu(G):F\supseteq A,\, G \text{ is open}\}
\] 
To be more precise in \cite{eidelman2004functional} such a measure is called strongly regular, but we follow most of the literature where  such a measure is called \textit{regular measure}.
\newline If $h_1,\,h_2 \in \R^m$, we denote for brevity
with $h_1h_2$ the scalar product $\langle h_1,h_2\rangle_{\R^m}$. 
Let us consider the following linear BSDE
\begin{align}\label{BSDEtrascinata}
p(t)&=\int_t^Tf(s)ds+\int_t^Tg(s)p(s)ds +\int_t^Tq(s)h(s)ds +\int_t^Tq(s)dW(s)  \\ \nonumber &+\int ^T_{t\vee(T-d)}\xi(s)\mu(ds).
\end{align}
Let $\Q$ be the probability measure, equivalent to the original one $\P$, such that
$$
\tilde W(t):= \int_0^th(s)\,ds+W(t)
$$
is a $\Q$-Wiener process.
Notice that $\frac{d \mathbb{Q}}{ d \mathbb{P}}= \mathcal{E}(\int_0^T h(s) \, d B_s)$ thus, thanks to hypothesis \ref{hyp:BSDEtrascinata}  the two measures are equivalent, and in particular for any $ \Theta \in L^1(\Omega, \mathbb{P})$
 \begin{equation}\label{stimaRadonNyk}\E_{\mathbb{Q}} |\Theta| \leq C \E|\Theta| \qquad \text{ and }\qquad\E |\Theta| \leq \tilde{C} \E_{\mathbb{Q} }|\Theta|\end{equation} for some $C, \tilde{C}$ depending only on $T$ and the process $h$.  
\newline We have the following formula for the unique solution of the linear BSDE (\ref{BSDEtrascinata}), 
which is the counterpart for the classical formula for the solution of a linear BSDE. 
We also notice that equation (\ref{BSDEtrascinata})
is a linear BSDE with a final datum $\xi$ acting not only at the final time $T$, but on the
whole interval $[T-d,T]$.
\begin{lemma}\label{lemma:BSDEtrascinata} Assume Hypothesis \ref{hyp:BSDEtrascinata} holds true, then the BSDE (\ref{BSDEtrascinata}) admits a unique adapted solution, that is a pair
of processes $(p,q)\in
\Lcal^2_{\mathcal{F}}(\Omega\times [0,T], \R^{n})\times 
\Lcal^2_{\mathcal{F}}(\Omega\times [0,T], \R^{n\times  m})$, satisfying the integral equation (\ref{BSDEtrascinata}).
The process $p $ is given by the formula
\begin{equation}\label{formula:BSDEtrascinata}
 p(t)=\E_\Q^{\Fcal_t}\left[ \int_{t\vee(T-d)}^Te^{\int_t^s g(u)du}\xi(s) \mu(ds)
 +\int_{t}^Te^{\int_t^s g(u)du}f(s)ds\right].
\end{equation}
\end{lemma}
\begin{proof} 
Under the probability measure $\Q$ equation (\ref{BSDEtrascinata}) can be rewritten as
\begin{equation}\label{BSDEtrascinataQ}
p(t)=\int_t^Tf(s)ds+\int_t^Tg(s)p(s)ds +\int_t^Tq(s)d\tilde W(s) +\int_{t\vee(T-d)}^T\xi(s) \mu(ds).
\end{equation}
Let us first prove that the process $p$ given by formula (\ref{formula:BSDEtrascinata}) verifies equation (\ref{BSDEtrascinata})
for $t\in[T-d,T]$. For $t\in[T-d,T]$ 
equation (\ref{BSDEtrascinataQ}) implies  
\begin{equation}\label{BSDEtrascinataQ-special}
p(t)=\E_\Q^{\Fcal_t}\left[\int_t^Tf(s)ds+\int_t^Tg(s)p(s)ds  +\int_{t}^T\xi(s) \mu(ds)\right].
\end{equation}
Using formula (\ref{formula:BSDEtrascinata}) we define
\begin{equation}\label{formula:BSDEtrascinataQ-special}
 \bar{p}(t): =\E_\Q^{\Fcal_t}\left[ \int_{t}^Te^{\int_t^s g(u)du}\xi(s)\mu(ds)
 +\int_{t}^Te^{\int_t^s g(u)du}f(s)ds \right], \quad t \in [T-d,T].
\end{equation}
It is immediate to see that $\bar p$ defined in formula (\ref{formula:BSDEtrascinataQ-special}) satisfies (\ref{BSDEtrascinataQ-special}), indeed putting formula (\ref{formula:BSDEtrascinataQ-special}) for $\bar p$ in equation (\ref{BSDEtrascinataQ-special}) we get
\begin{align}\label{formula:BSDEtrascinataQ-special1}
&\E_\Q^{\Fcal_t}\left[ \int_{t}^Te^{\int_t^s g(u)du}\xi(s)\mu(ds)
 +\int_{t}^Te^{\int_t^s g(u)du}f(s)ds\right]\\ \nonumber
 &=\E_\Q^{\Fcal_t}\left[\int_t^Tf(s)ds
 +\int_{t}^T\xi(s) \mu(ds)  \right]\\  \nonumber & +\int_t^Tg(s) \E_\Q^{\Fcal_s}\left[ \int_{s}^Te^{\int_s^r g(u)du}\xi(r) \mu(dr)
 +\int_{s}^Te^{\int_s^r g(u)du}f(r)dr ds \right] \\ \nonumber
  &= \E_\Q^{\Fcal_t}\left[\int_t^Tf(s)ds
  +\int_{t}^T\xi(s) \mu(ds) \right]\\ \nonumber & +\int_t^Tg(s)\left[ \int_{s}^Te^{\int_s^r g(u)du}\xi(r) \mu(dr)
  +\int_{s}^Te^{\int_s^r g(u)du}f(r)dr\right]ds. \\ \nonumber
\end{align}
Changing the order of integration 
it is immediate to see that $$\left(\E_\Q^{\Fcal_t}\left[ \int_{t}^Te^{\int_t^s g(u)du}\xi(s) \mu(ds)+\int_{t}^Te^{\int_t^s g(u)du}f(s)ds\right]\right)_{t\in[0,T]}$$ satisfies the integral equation (\ref{BSDEtrascinataQ-special}) that corresponds to equation \eqref{BSDEtrascinataQ}, by the usual application of the Martingale Representation Theorem.

Now let us consider the following equation, under $\mathbb{Q}$ again, for $t \in [0,T-d]$: 
\begin{equation}\label{BSDEtrascinataQT-d}
p(t)=\int_t^{T-d}f(s)ds+\int_t^{T-d}g(s)p(s)ds +\int_t^{T-d}q(s)d\tilde W(s)+ \bar{p}(T-d),
\end{equation}
that is a standard BSDE with final datum $\bar{p}(T-d)$ defined in \eqref{formula:BSDEtrascinataQ-special}. The unique solution of \eqref{BSDEtrascinataQT-d} is given by:
\begin{align}\label{formula:BSDEtrascinataQ-special2}
& \tilde{p}(t)=\E_\Q^{\Fcal_t}\left[ e^{\int_t^{T-d} g(u)du}\bar{p}(T-d)
 +\int_{t}^{T-d}e^{\int_t^s g(u)du}f(s)ds\right]\\ \nonumber
 &=\E_\Q^{\Fcal_t}\left[ e^{\int_t^{T-d} g(u)du}\E_\Q^{\Fcal_{T-d}}\left[ 
 \int_{T-d}^Te^{\int_{T-d}^s g(u)du}\xi(s)\mu(ds)\right.\right.\\ \nonumber
 &\left.\left.+\int_{T-d}^Te^{\int_{T-d}^s g(u)du}f(s)ds\right]+\int_{t}^{T-d}e^{\int_t^s g(u)du}f(s)ds\right],
\end{align}
Hence $p(t)$ defined by \eqref{formula:BSDEtrascinata}, can also be written as
\begin{equation}\label{soluzioneBSDEtrascinataQ}
p(t) := \left\{ \begin{array}{ll} \bar{p} (t) & t \in [T-d,T], \\ \\
\tilde{p}(t) & t \in [0,T-d)
\end{array}
\right.
\end{equation}
So by construction $p$ together with its corresponding martingale term $q$, that can be uniquely determined through the variation of  the process $p$ with the noise $W$, is a solution to \eqref{BSDEtrascinataQ}, and hence  a weak solution to \eqref{BSDEtrascinata} in the whole time interval $[0,T]$ .
\newline In particular, directly from \eqref{formula:BSDEtrascinata} we get that, see also \cite[Chapter 5]{ParRasbook},  for any $\beta>0$, taking also into account \eqref{stimaRadonNyk} and the Burkholder-Davis-Gundy inequalities:
\begin{align}\label{normaS:BSDEtrascinata}
&\E \sup_{t\in [0,T]}  e^{\beta t}|p(t)|^2   \leq C  \, \E_{\mathbb{Q}} \sup_{t\in [0,T]}  e^{\beta t}|p(t)|^2   \\
&   \leq 
C\, e^{\beta T} \E_{\mathbb{Q}} \left[   \sup_{t\in [0,T]}  \E_{\mathbb{Q}}^{\mathcal{F}_t} \Big| \int_{t}^Te^{\int_t^s g(u)du}\xi(s) \mu(ds)\Big|^2 \right.\\
&\left.\;+ \sup_{t\in [0,T]} 
\E_{\mathbb{Q}}^{\mathcal{F}_t} \Big|\int_{t}^Te^{\int_t^s g(u)du}f(s)ds\Big|^2 \right]
\nonumber\\ &
 \leq 
{C} \, e^{\beta T} \E_{\mathbb{Q}} \left[   \sup_{t\in [0,T]}  \E_{\mathbb{Q}}^{\mathcal{F}_t} \Big[ \int_{T-d}^T|\xi(s)| |\mu|(ds)\Big]^2 +   \frac{1}{\beta} 
 \sup_{t\in [0,T]}\E_{\mathbb{Q}}^{\mathcal{F}_t} \int_0^T  e^{\beta s } |f(s)|^2 \, ds \right] \nonumber
 \\&
\leq
C \, \E_{\mathbb{Q}} \left[ e^{\beta T}\sup_{t\in [0,T]}  \E _{\mathbb{Q}}^{\mathcal{F}_t}   \left( \int_{T-d}^T |\xi(s)|| \mu|(ds) \right)^2
+   \frac{1}{\beta}  \sup_{t\in [0,T]}   \E_{\mathbb{Q}}^{\mathcal{F}_t} \int_0^T  e^{\beta s } |f(s)|  ^2\, ds   \right]\nonumber 
\\& \leq 
C \, \left[ e^{\beta T} \left(\E _{\mathbb{Q}}  \int_{T-d}^T |\xi(s)| |\mu|(ds) \right)^2
+   \frac{1}{\beta}   \E _{\mathbb{Q}} \int_0^T  e^{\beta s }|f(s)|  ^2\, ds   \right] 
 \nonumber\\& \leq 
C \, \left[ e^{\beta T} \left(\E  \int_{T-d}^T |\xi(s)| |\mu|(ds) \right)^2
+   \frac{1}{\beta}   \E \int_0^T  e^{\beta s }|f(s)|  ^2\, ds   \right] \nonumber
\end{align}
where the constant $C$ may change from line to line but only depends on $T$ and $||g||_{\Lcal_{\mathcal{F}}^\infty(\Omega\times[0,T])}$, $||h||_{\Lcal_{\mathcal{F}}^\infty(\Omega\times[0,T])}$ and thus 
\begin{align}\label{norma:BSDEtrascinata1}
&\E \sup_{t\in [0,T]}  e^{\beta t}|p(t)|^2  \\ &\leq     C \left[ e^{\beta T}\left(|\mu|([T-d,T])\right)^2 \E \sup_{t\in [T-d,T]}|\xi(t)| ^2 +
 \frac{1}{\beta}    \E \int_0^T  e^{\beta s }|f(s)|  ^2\, ds  \right] \nonumber.
\end{align}
By the latter calculations together with standard considerations see for instance \cite{FuTessAOP2002}, using the Martingale Representation Theorem, we get for all $\beta >0$, 
\begin{align}\label{norma:BSDEtrascinata}
& \E \sup_{t\in [0,T]}  e^{\beta t}|p(t)|^2  +  \E \int_0^T e^{\beta s}\vert q(s)\vert^2\,ds  \\ 
&\leq c\left[ \E _{\mathbb{Q}}\sup_{t\in [0,T]}  e^{\beta t}|p(t)|^2  +  \E _{\mathbb{Q}}\int_0^T e^{\beta s}\vert q(s)\vert^2\,ds\right]\nonumber  \\ &\leq  c \left[ e^{\beta T} \left(\E  _{\mathbb{Q}}\int_{T-d}^T |\xi(s)| |\mu|(ds) \right)^2
+   \frac{1}{\beta}   \E  _{\mathbb{Q}} \int_0^T  e^{\beta s }|f(s)|  ^2\, ds   \right], \nonumber
 \\ &\leq  c \left[ e^{\beta T} \left(\E   \int_{T-d}^T |\xi(s)| |\mu|(ds) \right)^2
+   \frac{1}{\beta}   \E \int_0^T  e^{\beta s }|f(s)|  ^2\, ds   \right], \nonumber
     \end{align}
 and these estimates imply
\begin{align}\label{norma:BSDEtrascinata2}
&\E\int_0^T \frac{\beta}{2} e^{\beta s}\vert p(s)\vert^2\,ds  + \E\int_0^T e^{\beta s}\vert q(s)\vert^2\,ds \\
&\leq c  \left[ e^{\beta T}\left(|\mu|([T-d,T])\right)^2 \E \sup_{t\in [0,T]}|\xi(t)| ^2  +   \frac{1}{\beta} \E \int_0^T  e^{\beta s }|f(s)|  ^2\, ds  \right]\nonumber,
 \end{align}
where the constant $c$ depends on $||g||_{\Lcal_{\mathcal{F}}^\infty(\Omega\times[0,T])}$, $||h||_{\Lcal_{\mathcal{F}}^\infty(\Omega\times[0,T])}$, and $T$.
Therefore $ (p,q) \in \mathcal{S}^2_{\mathcal{F}}([0,T])  \times  \Lcal^2_{\mathcal{F}}(\Omega\times [0,T]; \R^{n\times m})$. Pathwise uniqueness follows by standard arguments since the non-classical terms disappears when one calculates the difference between solutions.  By the Yamada-Watanabe type result for weak solutions for BSDEs, see \cite{BuckEngeRasc2005}, also for BSDEs pathwise uniqueness implies the uniqueness in law; and the pathwise uniqueness together with the existence of the weak solution imply the existence of the strong solution. This concludes the proof.
\end{proof}

We apply the results collected in Lemma \ref{lemma:BSDEtrascinata} to prove existence and uniqueness of a solution of the following anticipated
ABSDE \eqref{ABSDEtrascinata}.
\newline As before, let $(W_t)_{t\geq 0}$ be a standard $m$-dimensional Brownian motion, and on the coefficients we make the following assumptions.

\begin{hypothesis}\label{hyp:ABSDEtrascinata}
 $f\in\Lcal^2_{\mathcal{F}}(\Omega\times [0,T], \R^{n})$, 
$g\in\Lcal^\infty_{\mathcal{F}}(\Omega\times [0,T+d], \R^{n\times  n})$
$h\in\Lcal^\infty_{\mathcal{F}}(\Omega\times [0,T+d], \R^{m})$
and $\xi\in\mathcal{B}^2_{\mathcal{F}}( [T-d,T])$ where $0<d \leq T$ is fixed. Let $\mu$ be a finite regular measure on $[T-d,T]$, and $\mu_1$ and $\mu_2$ be finite regular measures on $[-d,0]$.
\end{hypothesis}
\begin{remark}\label{rm-vectorvalued1}
The results in this Section can be extended from measures $\mu$, $\mu_1, \mu_2$ in Hypothesis \ref{hyp:ABSDEtrascinata} to vector valued finite regular measures, that allow to consider more general dependence on the past trajectory. 
\end{remark}
We will prove existence and uniqueness of the following anticipated BSDEs of backward type:
\begin{equation}\label{ABSDEtrascinata}
 \left\lbrace\begin{array}{l}
p(t)=\dis\int_t^Tf(s)ds+\dis\int_t^T\E^{\Fcal_s}\dis\int_{-d}^0g(s-\theta)p(s-\theta)\mu_1(d\theta) ds \\
\quad\quad
+\dis\int_t^T\E^{\Fcal_s}\dis\int_{-d}^0q(s-\theta)h(s-\theta)\mu_2(d\theta) ds+\dis\int_t^Tq_sdW_s \\\quad\quad+\int_{t\vee(T-d)}^T\xi(s) \mu(ds)  \\
p(T-\theta)=0, \; q(T-\theta)=0 \quad \forall \,\theta \in [-d,0 ).
 \end{array}
 \right.
\end{equation}
The ABSDE (\ref{ABSDEtrascinata}) is of the form of equation introduced in \cite{PengYang} and generalized in \cite{ElliottYang}, with the difference that it is given a final datum acting not only in $[T,T+d)$, but also in $[T-d,T]$, see also \cite{GMO}.
Notice that as soon as the process $q$ belongs to $  \Lcal^2_{\mathcal{F}}(\Omega\times [0,T+d], \R^{n\times  m})$
the term $\E^{\Fcal_\cdot}\dis\int_{-d}^0q(\cdot-\theta)h(\cdot-\theta)\mu_2(d\theta) $ has meaning since:
 \begin{align*}
 &\E \int_0^T \int_{-d}^0 |q(t-\theta)|^2 |h(t-\theta))|^2 |\mu_2 | (d\theta) \, dt \\
  &\quad\leq |\mu| ^2 ([-d,0]) ||h||^2_{\Lcal^\infty} \E \int_{0}^{T+d} |q(\rho)|^2  \, d\rho  
< +\infty
 \end{align*} 
\begin{theorem}\label{proposition:ABSDEtrascinata} Let Hypothesis \ref{hyp:ABSDEtrascinata} holds true. Then the ABSDE (\ref{ABSDEtrascinata}) admits a unique adapted solution, that is a pair of processes $(p,q)\in\Lcal^2_{\mathcal{F}}(\Omega\times [0,T+d], \R^{n})\times \Lcal^2_{\mathcal{F}}(\Omega\times [0,T+d], \R^{n\times  m})$, satisfying the integral equation (\ref{ABSDEtrascinata}).
\end{theorem}
\begin{proof}We notice that by the data of our problem, if it exists, the pair of processes $(p,q)$ solution
to the ABSDE (\ref{ABSDEtrascinata}), is such that $p(t)=q(t)=0$ for $t\in(T,T+d)$.

Let us consider the more general equation, for  $(\xi, \eta) \in  \Lcal^2_{\mathcal{F}}(\Omega;B([T-d,T+d];\R^n)) \times 
\Lcal^2_{\mathcal{P}}(\Omega\times [T,T+d], \R^{n\times  m})$

\begin{equation}\label{ABSDEtrascinatagen}
 \left\lbrace\begin{array}{l}
p(t)=\dis\int_t^Tf(s)ds+\dis\int_t^T\E^{\Fcal_s}\dis\int_{-d}^0g(s-\theta)p(s-\theta)\mu_1(d\theta) ds\\
\quad\quad +\dis\int_t^T
\E^{\Fcal_s}\dis\int_{-d}^0q(s-\theta)h(s-\theta)\mu_2(d\theta) ds +\dis\int_t^Tq(s)dW(s)\\
\qquad +\dis\int_{t\vee(T-d)}^T\xi(s) \mu(ds)  \\
p(T-\theta)= \xi(\theta) 
, \; q(T-\theta)= \eta(\theta) \quad \forall \,\theta \in [-d,0 ).
 \end{array}
 \right.
\end{equation}

\noindent We prove existence of a solution by a fixed point argument on the space $\Lcal^2_{\mathcal{F}}(\Omega\times [0,T+d], \R^{n})\times 
\Lcal^2_{\mathcal{F}}(\Omega\times [0,T+d], \R^{n\times  m})$ endowed with the equivalent norm \begin{equation}\label{norma-beta}
\Vert (p,q) \Vert _\beta= \E  \int_0^{T+d}   \vert p(s)\vert^2e^{\beta s}\,ds+   \int_0^{T+d} \vert q(s)\vert^2e^{\beta s}\,ds,
\end{equation}
with $\beta>0$ to be chosen in the following.
\newline Given $(y,z)\in \Lcal^2_{\mathcal{F}}(\Omega\times [0,T+d], \R^{n})\times 
\Lcal^2_{\mathcal{F}}(\Omega\times [0,T+d], \R^{n\times  m})$
we define the map $\Gamma :\Lcal^2_{\mathcal{F}}(\Omega\times [0,T+d], \R^{n})\times 
\Lcal^2_{\mathcal{F}}(\Omega\times [0,T+d], \R^{n\times  m})\rightarrow
 \Lcal^2_{\mathcal{F}}(\Omega\times [0,T+d], \R^{n})\times \Lcal^2_{\mathcal{F}}(\Omega\times [0,T+d], \R^{n\times  m})$.
\newline The pair $(p,q):=\Gamma (y,z)$ is given by
the pair of processes solution of the following BSDE given in integral form:
\begin{equation}\label{BSDE-gamma}
 \left\lbrace\begin{array}{l}
p(t)=\dis\int_t^Tf(s)ds+\dis\int_t^T\E^{\Fcal_s}\dis\int_{-d}^0g(s-\theta)y(s-\theta)\mu_1(d\theta) ds \\ \qquad  +\dis\int_t^T
\E^{\Fcal_s}\dis\int_{-d}^0z(s-\theta)h(s-\theta)\mu_2(d\theta) ds\\
\quad\quad+\dis\int_t^Tq(s)dW_s +\int_{t\vee(T-d)}^T\xi(s) \mu(ds)  \\
p(T-\theta)= \xi(\theta) 
, \; q(T-\theta)= \eta(\theta )\quad \forall \,\theta \in [-d,0 ).
 \end{array}
 \right.
\end{equation}
Thanks to Lemma \ref{lemma:BSDEtrascinata}  it turns out that $(p,q)
\in \Lcal^2_{\mathcal{F}}(\Omega\times [0,T], \R^{n})\times 
\Lcal^2_{\mathcal{F}}(\Omega\times [0,T], \R^{n\times  m})$, and together with the condition given in (\ref{BSDE-gamma}) it turns out that $(p,q)\in\Lcal^2_{\mathcal{F}}(\Omega\times [0,T+d], \R^{n})\times 
\Lcal^2_{\mathcal{F}}(\Omega\times [0,T+d], \R^{n\times  m})$.
So $\Gamma$ is well defined. Next we prove that $\Gamma$ is a contraction.
Let $y,\bar y \in \Lcal^2_{\mathcal{F}}(\Omega\times [0,T+d], \R^{n}) $ and $z, \bar z \in 
\Lcal^2_{\mathcal{F}}(\Omega\times [0,T+d], \R^{n\times  m})$, and set $\hat y=y-\bar y,\, \hat z=z-\bar z$. We denote
$(\hat p, \hat q)= \Gamma(y,z)-\Gamma (\bar y,\bar z).$ So 
 \begin{multline}\label{BSDE-gamma-difference}
\hat p(t)=\dis\int_t^Tf(s)ds+\dis\int_t^T\E^{\Fcal_s}\dis\int_{-d}^0g(s-\theta)\hat y(s-\theta)\mu_1(d\theta) ds \\ +\dis\int_t^T
\E^{\Fcal_s}\dis\int_{-d}^0\hat z(s-\theta)h(s-\theta)\mu_2(d\theta) ds+\dis\int_t^T\hat q(s)dW_s .
 \end{multline}
 Equation (\ref{BSDE-gamma-difference}) is a special case of the BSDE (\ref{BSDEtrascinata}), whose existence and uniqueness
 have been studied in Lemma \ref{lemma:BSDEtrascinata}. 
 By estimate \eqref{norma:BSDEtrascinata1} we get (here and in the following $c$ is a constant whose value can change from line to line)
 \begin{align}\label{norma:ABSDEtrascinataGamma}
 &\E\int_0^T \left( \frac{\beta}{2}\vert \hat p(s)\vert^2+\vert \hat q(s)\vert^2\right)e^{\beta s}\,ds\\ \nonumber 
 &\leq \frac{2c}{\beta}\E\int_0^T \vert\int_{-d}^0 g(s-\theta)\hat y(s-\theta)\mu_1(d\theta) ds +
\dis\int_{-d}^0\hat  z(s-\theta)h(s-\theta)\mu_2(d\theta)\vert^2 ds\\ \nonumber
& \leq \frac{2c}{\beta}\left\lbrace\E\int_0^T \int_{-d}^0 \vert\hat y(s-\theta)\vert^2 \vert\mu_1\vert(d\theta) ds +\E\int_0^T
\int_{-d}^0\vert\hat z(s-\theta)\vert^2 \vert\mu_2\vert(d\theta) ds\right\rbrace\\ \nonumber
& \leq \frac{2c}{\beta}\left\lbrace\int_{-d}^0\left[\E\int_0^T  \vert\hat y(s-\theta)\vert^2 ds\right] \vert\mu_1\vert(d\theta) \right.\\ \nonumber&\qquad\left.+
\int_{-d}^0\left[\E\int_t^T
\vert\hat z(s-\theta)\vert^2 ds\right]\vert\mu_2\vert(d\theta) \right\rbrace\\ \nonumber
&= \frac{2c}{\beta}\left\lbrace\int_{-d}^0\left[\E\int_0^T  \vert\hat y(s)\vert^2 ds\right] \vert\mu_1\vert(d\theta) +\int_{-d}^0\left[\E\int_t^T
\vert\hat z(s)\vert^2 ds\right]\vert\mu_2\vert(d\theta) \right\rbrace\\ \nonumber
&\leq c\frac{2}{\beta}\Vert (\hat y,\hat z)\Vert_\beta.
\end{align}
By choosing $\beta>0$ such that $c\frac{2}{\beta}<1$ we have proved that $\Gamma $ is a contraction, and its unique fixed point is the unique solution of the ABSDE
(\ref{ABSDEtrascinata}).
\end{proof}
Equation (\ref{ABSDEtrascinata}) can be written in differential form if we make some additional assumptions on the measure $\mu$.
If we assume that
\[
 \mu=c\delta_T+\tilde \mu,\; c\in\R,
\]
where $\tilde \mu$ is a measure on $(T-d,T)$ absolutely continuous with respect to the Lebesgue measure, equation (\ref{ABSDEtrascinata})
can be written in differential form as
\begin{equation}\label{ABSDEtrascinata-formadiff}
 \left\lbrace\begin{array}{l}
-dp(t)=f(t)dt+\E^{\Fcal_t}\dis\int_{-d}^0g(t-\theta)p(t-\theta)\mu_1(d\theta) dt  \\ +
\E^{\Fcal_t}\dis\int_{-d}^0q(t-\theta)h(t-\theta)\mu_2(d\theta) dt 
+q(t)dW(t)+\xi(t) \eta_{\tilde \mu}(t) dt\\
p(T)=c\xi(T),\;p(T-\theta)=0, \; q(T-\theta)=0 \quad \forall \,\theta \in [-d,0 ),
 \end{array}
 \right.
\end{equation}
where for $t\in[T-d,T]$, $\eta_{\tilde \mu}$ is the Radon-Nikodym derivative of $\tilde \mu$ with respect to the Lebesgue measure,
that is $\eta_{\tilde \mu}$ is defined by the relation
\[
\tilde \mu(dt)=\eta_{\tilde \mu}(t)dt.
\]
If we do not make additional assumptions on the measure $\mu$, the differential form of equation (\ref{ABSDEtrascinata})
does not make sense, since the term 
\[
 d \int_{t\vee(T-d)}^T\xi(s) \mu(ds) 
 \]
is not well defined for $t\in [T-d,T]$.
In the following, we build an approximating ABSDE whose differential form makes sense. This approximating ABSDE is obtained by a suitable approximation of $\mu$: the construction of this sequence of approximating measures $(\mu_n)_{n\geq 1}$ is given in the following Lemma, which is the analogous of Lemma 5 in \cite{GMO}. In the following, given $I\subset \R$ we denote with $C_b(I,\R])$ the space of bounded an continuous functions from $I$ to $\R$; with the notation $\lambda_{[T-d,T]}$, we denote the Lebesgue measure on $[T-d,T]$.
\begin{lemma}\label{lemma:approx-measure}
 Let $\bar\mu$ be a finite regular measure on $[T-d,T]$, such that $\bar \mu (\left\lbrace T \right\rbrace)=0$.
There exists a sequence $(\bar\mu_n)_{n\geq 1}$ of finite regular measures on $[T-d,T]$, absolutely continuous with respect to $\lambda_{[T-d,T]}$ and such that
 \begin{equation}\label{conv-a-mu}
  \bar\mu=\lim_{n\rightarrow\infty}\bar \mu_n, \quad \text{ in the sense of the narrow convergene,}
 \end{equation}
that is for every $f\in C_b([T-d,T],\R)$
\begin{equation}\label{conv-a-mu:expl}
  \int_{T-d}^T f\,d\bar\mu=\lim_{n\rightarrow\infty}\int_{T-d}^T f\, d\bar\mu_n
 \end{equation}
\end{lemma}
Notice that we can apply the previous Lemma to the approximation of the measure $\mu$ by defining $\bar \mu$ such that for any $A\in \calb ([T-d,T])$
 \begin{equation}\label{bar-mu}
   \bar\mu(A)=\mu(A\backslash\left\lbrace T\right\rbrace):
 \end{equation}
the measure $\bar \mu$ is obtained from the original measure $\mu$, by subtracting to $\mu$ its mass in
 $\left\lbrace T\right\rbrace$.
 Lemma \ref{lemma:approx-measure} ensures that there exists a sequence of measures $(\bar\mu_n)_{n\geq 1}$, on $[T-d,T]$,
 which are absolutely continuous with respect to the Lebesgue measure on $[T-d,T]$ and converge to $\bar \mu$.

The next step is to build an approximation, in a sense that we are going to precise, of the equation $(\ref{ABSDEtrascinata})$, by approximating 
$\bar\mu $ obtained by $\mu$ in (\ref{bar-mu}).
\begin{proposition}\label{prop:approxABSDEtrascinata}
 Let Hypothesis \ref{hyp:ABSDEtrascinata} holds true and assume $\xi \in \mathcal{S}^2_{\mathcal F }([T-d,T]),$ let $\bar\mu$ be defined by (\ref{bar-mu}), and let us consider $(\bar\mu_n)_n$ the approximations of $\bar\mu$, absolutely continuous with respect to the Lebesgue measure on $(T-d,T)$.
Let us consider the approximating ABSDEs (of ``standard'' type):
 \begin{equation}\label{ABSDEtrascinata-approx-mu}
 \left\lbrace\begin{array}{l}
p^n(t)=\dis\int_t^Tf(s)ds+\dis\int_t^T\E^{\Fcal_s}\dis\int_{-d}^0g(s-\theta)p^n(s-\theta)\mu_1(d\theta) ds \\
\quad\quad\quad+\dis\int_t^T
\E^{\Fcal_s}\dis\int_{-d}^0q^n(s-\theta)h(s-\theta)\mu_2(d\theta) ds \\ \qquad\quad +\dis\int_t^Tq^n(s)dW_s +\int_{t\vee(T-d)}^T\xi(s) \bar\mu_n(ds)+ \mu({T})\xi(T),  \\
p^n(T-\theta)=0, \; q^n(T-\theta)=0 \quad \forall \,\theta \in [-d,0 ).
 \end{array}
 \right.
 \end{equation}
Then the pair $(p^n,q^n)$, solution to (\ref{ABSDEtrascinata-approx-mu}) converges
in $\Lcal^2_{\mathcal{F}}(\Omega\times [0,T+d], \R^{n})\times 
\Lcal^2_{\mathcal{F}}(\Omega\times [0,T+d], \R^{n\times  m})$ to the pair $(p,q)$ solution to
(\ref{ABSDEtrascinata}).
\end{proposition}
\begin{proof}Let us first prove that the sequence $(p^n,q^n)_n$ is a Cauchy sequence in
$\Lcal^2_{\mathcal{F}}(\Omega\times [0,T+d], \R^{n})\times 
\Lcal^2_{\mathcal{F}}(\Omega\times [0,T+d], \R^{n\times  m})$.
The equation satisfied by $(p^n(t)-p^k(t),q^n(t)-q^k(t)), \,n,k\geq 1$, turns out to be the following ABSDE:
\begin{equation}\label{ABSDEtrascinata-approx-nk}
 \left\lbrace\begin{array}{l}
p^n(t)-p^k(t)=\dis\int_t^T\E^{\Fcal_s}\dis\int_{-d}^0g(s-\theta)\left(p^n(s-\theta)-p^k(s-\theta)\right)\mu_1(d\theta) ds\\
\qquad\qquad\qquad +\dis\int_t^T
\E^{\Fcal_s}\dis\int_{-d}^0\left(q^n(s-\theta)-q^k(s-\theta)\right)h(s-\theta)\mu_2(d\theta) ds \\
\qquad\qquad\qquad+\dis\int_t^T\left(q^n(s)-q^k(s)\right)dW(s) +\int_{t\vee(T-d)}^T\xi(s) \bar\mu_n(ds)  \\ \qquad\qquad\qquad \displaystyle - \int_{t\vee(T-d)}^T\xi(s) \bar\mu_k(ds),  \\ \\
p^n(T-\theta)-p^k(T-\theta)=0, \; q^n(T-\theta)-q^k(T-\theta)=0 \quad \forall \,\theta \in [-d,0 ).
 \end{array}
 \right.
 \end{equation}
Notice that the terms $\int_{t\vee(T-d)}^T\xi(s) \bar\mu_n(ds) ,\, \int_{t\vee(T-d)}^T\xi(s) \bar\mu_k(ds),$ are Ito terms, so equation (\ref{ABSDEtrascinata-approx-nk}) is a standard ABSDE. By standard estimates, see e.g. \cite{ElliottYang}, Lemma 2.3, formula (3), as $n,k\rightarrow\infty $
 \begin{multline*}
\E \vert p^n(t)-p^k(t)\vert^2+\E\int_t^T  \vert p^n(s)-p^k(s)\vert ^2\,ds +\E\int_t^T
\vert q^n(s)-q^k(s)\vert\ ^2 ds \\\leq \E \left\vert \int_{t\vee(T-d)}^T \xi(s)( \bar\mu_n(ds)-
\bar\mu_k(ds))\right\vert^2\rightarrow 0  
 \end{multline*}
 by the narrow convergence of the sequence of measures $\bar \mu_n$. So the sequence $(p^n, q^n)$ is a Cauchy sequence in $\Lcal^2_{\mathcal{F}}(\Omega\times [0,T+d], \R^{n})\times \Lcal^2_{\mathcal{F}}(\Omega\times [0,T+d], \R^{n\times  m})$. It remains to show that it converges to $(p,q)$ solution of equation (\ref{ABSDEtrascinata}). Let us denote 
$$
(\bar p, \bar q)= \lim_{n\rightarrow \infty } (p^n,q^n) \text{ in } \Lcal^2_{\mathcal{F}}(\Omega\times [0,T+d], \R^{n})\times 
\Lcal^2_{\mathcal{F}}(\Omega\times [0,T+d], \R^{n\times  m}).
$$
and  for every $t \in [0,T]$:
$$
\lim_{n \to +\infty} \E \vert p^n(t)-\bar{p}(t)\vert^2 =0
$$
Notice also that $p^{n}(s-\theta), \bar{p}(s-\theta)=0,  \text{ for } s-\theta > T$: thus we have that, for all $ t \in [0,T]$,
\begin{align*}
&\E\left|\int_t^T\dis\int_{-d}^0g(s-\theta)p^{n}(s-\theta)\mu_1(d\theta) ds - \int_t^T\dis\int_{-d}^0g(s-\theta) \bar{p}(s-\theta)\mu_1(d\theta) ds \right|^2 \\
&=\E\left|\dis\int_{-d}^0\int_t^Tg(s-\theta) \left(p^{n}(s-\theta)- \bar{p}(s-\theta)\right) ds \mu_1(d\theta) \right|^2 \\
&\leq  |g|^2_{\Lcal^\infty[0,T]}\E\left(\dis\int_{-d}^0\int_t^T\vert p^{n}(s-\theta)- \bar{p}(s-\theta)\vert ds\vert  \mu_1\vert(d\theta) \right)^2 \\
&\leq  |g|^2_{\Lcal^\infty[0,T]}\E\left(\dis\int_{-d}^0\int_{(t+\theta)\vee 0}^{T+\theta}\vert p^{n}(\sigma)- \bar{p}(\sigma)\vert d\sigma\vert \mu_1\vert(d\theta)\right)^2 \\
&\leq  |g|^2_{\Lcal^\infty[0,T]}\E\left(\dis\int_{-d}^0\int_{ 0}^{T}\vert p^{n}(\sigma)- \bar{p}(\sigma)\vert d\sigma\vert   \mu_1\vert(d\theta) \right)^2 \\
&\leq  |g|^2_{\Lcal^\infty[0,T]}\E\left(\dis\int_{-d}^0 \vert\mu_1\vert(d\theta)\int_{ 0}^{T}\vert p^{n}(\sigma)- \bar{p}(\sigma)\vert d\sigma  \right)^2 \\
&\leq  |g|^2_{\Lcal^\infty[0,T]}\left(\dis\int_{-d}^0 \vert\mu_1\vert(d\theta)\right)^2\E\left(\int_{ 0}^{T}\vert p^{n}(\sigma)- \bar{p}(\sigma)\vert d\sigma  \right)^2 \\
&\leq T |g|^2_{\Lcal^\infty[0,T]}\vert\mu_1\vert([-d,0])^2 \E \int_0^T |p^n(r) - \bar{p}(r)|^2 \, dr .
\end{align*}
So
\begin{align*}
\E&\left|\int_t^T\dis\int_{-d}^0 q^{n}(s-\theta)h(s-\theta)
\mu_2(d\theta) ds - \int_t^T\dis\int_{-d}^0 \bar{q}(s-\theta)h(s-\theta)\mu_2(d\theta) ds \right|^2 \\
&\leq T |h|^2_{\Lcal^\infty[0,T]}\vert\mu_2\vert([-d,0])^2 \E \int_0^T |q^{n}(r) - \bar{q}(r)|^2 \, dr 
\end{align*}
By passing to the limit as $n\rightarrow \infty$ in equation (\ref{ABSDEtrascinata-approx-mu}) we get
\begin{equation*}
 \left\lbrace\begin{array}{l}
\E^{\Fcal_t}\bar p(t)=\E^{\Fcal_t}\dis\int_t^Tf(s)ds+
\E^{\Fcal_t}\dis\int_t^T\dis\int_{-d}^0g(s-\theta)\bar p(s-\theta)\mu_1(d\theta) ds \\
\quad\quad \quad\quad +
\E^{\Fcal_t}\dis\int_t^T \dis\int_{-d}^0\bar q(s-\theta)h(s-\theta)\mu_2(d\theta) ds \\ \qquad\qquad +
\displaystyle \E^{\Fcal_t}\left[\int_{t\vee(T-d)}^T\xi(s) \bar\mu(ds)
+ \mu({T})\xi_T\right],  \\ \\
\bar p(T-\theta), \; \bar q(T-\theta)=0 \quad \forall \,\theta \in [-d,0 ).
 \end{array}
 \right.
 \end{equation*}
It follows immediately that $(\bar p, \bar q)=(p,q)$ and this concludes the proof.
 \end{proof}

 \section{The controlled problem and the stochastic maximum principle}
 \label{Sec:contr}
Let us consider the following controlled state equation in $\R^n$ 
 \begin{equation}\label{eq:state:fin}
\begin{system}
dx(t) = f(t,x_t,u_t)\, dt + g( t,x_t,u_t)dW(t), \qquad t \in [0,T],\\
x(\theta) = \bar x(\theta), \quad   u(\theta) = \eta (\theta), \qquad \theta \in [-d, 0],
\end{system}
\end{equation}
where $W$ in this section, for simplicity of notation, will be supposed to be a real  standard Brownian motion, and $x_{t}$ and  $u_{t}$ denote the past trajectories from time $t-d$ up to time $t$.
Moreover $\bar x$  and $ \eta$ are  the initial paths of the state and of the control respectively, and we assume $\eta$ to be deterministic and such that $\int_{-d}^0 \eta ^2 (t) \, dt < +\infty$.
By admissible control we mean an $\Fcal_t$-progressively measurable process with values in a convex set  $U\subset\R^k$. 
\begin{equation}\label{eq:admissible_control}
\E \int^T_{-d}\abs{u(t)}^2 \, dt < \infty,
\end{equation}
such that $ u(\theta)= \eta(\theta), \ \mathbb{P}-\text{a.s. for a.e. } \theta\in [-d,0]$. We will denote this space of admissible controls by $\mathcal{U}$.

We want to minimize the following cost functional 
\begin{equation}\label{costo:fin}
J(u(\cdot)) = \E \int_0^T l(t,x_t, u_t) dt + \E\, h(x_T) 
\end{equation}
over all admissible controls.
We make the following assumptions on $f,\,g,\,l,\,h$ and on the initial condition $\bar x$.
Here and in the following we denote by $E=C_b([-d,0],\R^n)$ and $K=C_b([-d,0],\R^k)$.

\begin{hypothesis}\label{ip:findim} 
Let $\mu_i,\, i=1,...,6$ and $\mu$ be finite regular measures. 
\newline We assume that $f$, $g$, $l $ and $ h$ are defined for any $x \in E$ and any $u \in K$ in terms of  $\bar{f}:\Omega\times [0,T]\times  \R^{n}\times \R^{k} \rightarrow \R^n,\,\bar{g}:\Omega\times[0,T]\times \R^{n}\times \R^{k}\rightarrow \R^n$,  $\bar{l}:[0,T]\times \R^{n}\times \R^{k} \rightarrow \R$  and $\bar{h}:\R^{n}\rightarrow \R$
as follows 
 $$ f(t,x,u)=\bar f (t, \int_{-d}^0 x(\theta )\mu_1(d\theta),  \int_{-d}^0 u(\theta )\mu_3(d\theta)),$$
  $$ g(t,x,u)=\bar g (t, \int_{-d}^0 x(\theta )\mu_2(d\theta),  \int_{-d}^0 u(\theta )\mu_4(d\theta)),$$
  $$ l(t,x,u)=\bar l (t,  \int_{-d}^0 x(\theta )\mu_5(d\theta),  \int_{-d}^0 u(\theta )\mu_6(d\theta)),$$
  $$ h(x)=\bar h \left( \int_{-d}^0 x(\theta )\mu(d\theta)\right).$$
Here and in the following in the drift $f$, and correspondingly in $\bar f$, and in the diffusion term $g$, and correspondingly in $\bar g$,  we omit the dependence on $\omega$. We assume that for each $\Fcal_t, \, t\in [0,T]$-adapted processes, $x,u \in \Lcal_\Fcal^2(\Omega, B([0,T;\mathbb{R}^k]) $, the processes $\bar f (\cdot, x_\cdot, u_\cdot),\,\bar g (\cdot, x_\cdot, u_\cdot)$ are also  $\Fcal_t, \, t\in [0,T]$-adapted.
.
\newline We will assume that  $\bar{f}, \,\bar{g}$ and $\bar{l}$ are Borel measurable and
differentiable with respect to the second and to the third variable, that with an abuse of notation we still refer to as $x$ and $u$. 
Moreover $\bar{f}_x, \bar{f}_u,\, \bar{g}_x$ and $ \bar{g}_u$ are uniformly bounded, while $\bar{l}_x, \,\bar{l}_u$ have linear growth with respect to $x$ and $u$, uniformly in $t$, finally $\bar{h}$  is differentiable and $\bar{h}_x$ has linear growth too.
Moreover we will use the following notations
$$\bar f_x (t, \int_{-d}^0 x(\theta )\mu_1(d\theta),  \int_{-d}^0 u(\theta )\mu_3(d\theta))=:\bar f_x (t,x,u),$$
$$\bar f_u (t, \int_{-d}^0 x(\theta )\mu_1(d\theta),  \int_{-d}^0 u(\theta )\mu_3(d\theta))=:\bar f_u (t,x,u)$$
  $$ \bar g_x (t, \int_{-d}^0 x(\theta )\mu_2(d\theta),  \int_{-d}^0 u(\theta )\mu_4(d\theta))=:\bar g_x (t,x,u),$$ $$\bar g_u (t, \int_{-d}^0 x(\theta )\mu_2(d\theta),  \int_{-d}^0 u(\theta )\mu_4(d\theta)=:\bar g_u (t,x,u)$$
  $$\bar l_x (t,  \int_{-d}^0 x(\theta )\mu_5(d\theta),  \int_{-d}^0 u(\theta )\mu_6(d\theta))=:\bar l_x (t,x,u),$$ $$\bar l_u (t,  \int_{-d}^0 x(\theta )\mu_5(d\theta),  \int_{-d}^0 u(\theta )\mu_6(d\theta))=:\bar l_u (t,x,u)$$
  $$\bar h_x \left( \int_{-d}^0 x(\theta )\mu(d\theta)\right)=:\bar h_x (x).$$
\end{hypothesis}
\begin{remark}\label{rm-vectorvalued}
In analogy to Remark \ref{rm-vectorvalued1}, all the results in this Section and throughout the paper can be extended to measures $\mu_i,\, i=1,...6,\, \mu$ considered in Hypothesis \ref{ip:findim}  possibly vector valued finite regular measures with values respectively in $\R^{j_i},\, i=1,...,6$, $\R^j$, with $j_i,j\geq 1,\,i=1,..,6$. 
\end{remark}
We notice that when the coefficients are stochastic, under Hypothesis \ref{ip:findim} existence and uniqueness of a solution of equation \eqref{eq:state:fin} holds true, see e.g. \cite{mohammed1998stochastic}, Theorem I.1

We notice that the terms 
\begin{equation*}
 \int_{-d}^0 u(\theta )\mu_3(d\theta),\quad  \int_{-d}^0 u(\theta )\mu_4(d\theta),\quad \int_{-d}^0 u(\theta )\mu_6(d\theta)
\end{equation*}
appearing respectively in the drift $f$, in the diffusion $g$ and in the current cost $l$ do not make sense in a standard way and for every $t\in[0,T]$ as soon as the control $u$ is not  assumed to be integrable  with respect to the measures $\mu_3, \mu_4, \mu_6 $, but only square integrable with respect to the Lebesgue measure in $[-d,0]$. 
\newline So it is necessary to give a precise meaning to the state equation and to the current cost.
 First of all we want to clarify that for any $ u \in \mathcal{U}$  equation \eqref{eq:state:fin} is well defined, 
 indeed for any $ u \in \mathcal {U}$ and any finite regular measure $\tilde\mu$ we have that:
 \begin{align*}
 \E \int_0^T \int_{-d}^0 |u(t+\theta)|^2 |\tilde\mu| (d\theta) \, dt  \leq |\tilde\mu| ([-d,0])  \E \int_{-d}^{T} |u(\rho)|^2  \, d\rho  
< +\infty
 \end{align*} 
thus 
\begin{equation*}
\int_{-d}^0 |u(t+\theta)|^2 |\tilde\mu|   (d\theta) 
< +\infty, \quad \text{ a.s. for a.e.    } t \in [0,T].
\end{equation*}
Then we can deduce that, thanks to Hypothesis \ref{ip:findim}, for all $x \in \mathcal{S}^p_{\mathcal F}([-d,T])$  and $u \in \mathcal{U}$ the processes are square integrable: $f(t,x_t,u_t) \in \Lcal^2_{\mathcal{F}}(\Omega\times [0,T]; \R^n), \,g(t,x_t,u_t) \in    \Lcal^2_{\mathcal{F}}(\Omega\times [0,T]; \R^{n})$.
In a similar way it follows that the current cost is well defined.
\newline Moreover for any $u \in \mathcal {U}$, there exists a solution $ x= x^u \in \mathcal{S}^p_{\mathcal F}([-d,T]) $: the result follows in the same way as for controlled stochastic delay equations without delay in the control, and it is substantially cointained e.g. in \cite{mohammed1998stochastic}, where stochastic delay equations with random drift and diffusion are solved.

Next we want to show that the adjoint equation of a delay equation is of the form of ABSDE (\ref{ABSDEtrascinata}), and it allows to formulate a stochastic maximum principle for finite dimensional controlled state equations with delay, and in the case of final cost functional depending on the history of the process.
\newline Many recent papers, see e.g. \cite{ChenWuAutomatica2010}, \cite{ChenWuYu2012}, deal with similar problems, but only in the simpler case of final cost functionals not depending on the past of the process. Moreover only the case of pointwise delay is considered, or in some cases the past affects the system at time $t$ by terms of the form
\[
\int_{-d}^0e^{-\lambda \theta }\xi(t+\theta)\,d\theta
\]
where $\xi$ may coincide with the state $x$ of the system, and/or with the control $u$.
These two choices coincide respectively with taking the measures $\mu_i,\,i=1,...,6$ delta Dirac measures and measures absolutely continuous with respect to the Lebesgue measure and with exponential density.
\newline In the present paper we are able to handle $\mu_i,\,i=1,..., 6$ finite regular measures on $[-d,0]$: such a general case is treated in the paper \cite{hu1990maximum}, only in the case without delay in the control and it is here proved by means of anticipated BSDEs.
 
In order to write the adjoint equation, at first we study the variation of the state: let us consider the pair $(x,u)$, where $x$ is solution to equation (\ref{eq:state:fin}) and $u$ is the control process in this equation, and let $v \in \mathcal{U}$ be another admissible control; set $ \bar v= v-\bar u$ and
\begin{equation}\label{urho}
u^\rho=\bar u+\rho \bar v.
\end{equation}
Also $u^\rho$ turns out to be an admissible control. Let $x^\rho$ be the solution of equation (\ref{eq:state:fin})
corresponding to the admissible control $u^\rho$ and let $y$ be the solution of the following linear equation
\begin{equation}\label{eq:state:fin:variation}
\begin{system}
dy(t) = \displaystyle \int_{-d}^0 \bar{f}_x(t,\bar x_t,\bar u_t)y_t (\theta) \mu_1(d\theta)\,dt +  \int_{-d}^0\bar{f}_u(t,\bar x_t,\bar u_t)\bar v_t (\theta) \mu_3(d\theta)\,dt+
 \\ \qquad\quad +  \displaystyle \int_{-d}^0 \bar{g}_x(t,\bar x_t,\bar u_t)y_t (\theta) \mu_2(d\theta)\,dW_t +  \int_{-d}^0 \bar{g}_u(t,\bar x_t,\bar u_t)\bar v_t (\theta) \mu_4(d\theta)\,d W_t\\ \\
y(\theta) = 0,\,\forall\, \theta\in [-d,0].
\end{system}
\end{equation}
 With an immediate extension of Theorem 3.2 in \cite{hu1996maximum} to the case with delay in the control, we have the following first order expansion
\begin{equation}\label{resto:fin:rho}
 x^\rho(t)=\bar x(t)+\rho y(t)+R^\rho(t),\, t\in[0,T],\quad \lim_{\rho\to 0}\frac{1}{\rho^2}\E\sup_{t\in[0,T]}\vert R^\rho(t)\vert^2=0.
\end{equation}
We are going to prove that equation (\ref{ABSDEtrascinata}) with
\begin{align}\label{fghxi-espl}
&f(t)=\E^{\Fcal_t}\dis\int_{-d}^0{\bar l_x(t-\theta, \bar x(t-\theta)\bar u(t-\theta))}(d\theta),
\\ \nonumber & g(t)=\bar{f}_x(t, \bar x_t,\bar u_t),\quad  h(t)= \bar{g}_x(t, \bar x_t,\bar u_t), \quad\xi(t)= h_x(\bar x_T) .
\end{align}
is the adjoint equation in the control problem with cost functional (\ref{costo:fin}).
We notice with the coefficients given by \eqref{fghxi-espl} the BSDE \eqref{ABSDEtrascinata} is solvable by Theorem \ref{proposition:ABSDEtrascinata} since Hypothesis \ref{hyp:ABSDEtrascinata} is satisfied.
\newline To prove that \eqref{ABSDEtrascinata} is the adjoint equation, for a.a. $\tau \in[0,T]$, $ x\in E, \ u \in  L^2([-d,T];U),\, p, q \in \R^n$, we define the Hamiltonian function as 
\begin{equation}\label{hamiltonian:fin}
\begin{split}
\mathcal{H}(t,x,u,p,q) &=
\bar{f}\left(\tau,\int_{-d}^0x(\theta)\mu_1(d\theta), \int_{-d}^0u(\theta)\mu_3(d\theta)\right)p\\
&+\bar{g}\left(\tau,\int_{-d}^0x(\theta)\mu_2(d\theta),\int_{-d}^0u(\theta)\mu_4(d\theta)\right)q\\&+ \bar{l}\left(\tau,\int_{-d}^0x(\theta)\mu_5(d\theta),\int_{-d}^0u(\theta)\mu_6(d\theta)\right)
\\
&=\bar{f}\left(\tau,x, u\right)p+\bar{g}\left(\tau,x, u\right)q +
l(\tau,x,u),
\end{split}
\end{equation}
where the last expression will be used, with an abuse of notation, to shorten the formulas.
Notice that the Hamiltonian function is not defined for every $\tau$, as discussed at the beginning of this Section, due to the fact that $\bar f$,  $\bar g$ and $\bar l$ depend respectively on the terms $\displaystyle\int_{-d}^0u(t-\theta+\eta)\mu_3(d\eta)$, $\displaystyle\int_{-d}^0u(t-\theta+\eta)\mu_4(d\eta)$ and $\displaystyle\int_{-d}^0u(t-\theta+\eta)\mu_6(d\eta)$. 
\newline  The Hamiltonian function turns out to be a $p$ -integrable function in time for any $p\geq 1 $, and so for any function $v\in L^q([0,T])$ the integral 
\[
\int_0^T\mathcal{H}(t,x,u,p,q)v(t)\,dt, \quad t \in [0,T], \ x \in E, u \in L^2([-d,T];U), p,q \in \R^n.
\]
makes sense, and this integral appears in the proof of the stochastic maximum principle, see the next Theorem on the stochastic maximum principle.

In the formulation of the stochastic maximum principle, the adjoint ABSDE turns out to be nothing else than equation \eqref{ABSDEtrascinata}, with  with $f,g,h$ and $\xi$ given in (\ref{fghxi-espl}).

\begin{theorem}\label{maxprinc-findim}
Let Hypothesis \ref{ip:findim} holds true. Let $(p,q)$ be the unique solution of the ABSDE
 \begin{equation}\label{ABSDEtrascinata-aggiunta-mu}
  \left\lbrace\begin{array}{l}
p(t)=\dis\int_t^T \E^{\Fcal_s}\dis\int_{-d}^0{\bar l_x\left(s-\theta , \bar x(s-\theta),\bar u(s-\theta)\right)}
\mu_5(d\theta)\,ds\\
\qquad+\dis\int_t^T\E^{\Fcal_s}\dis\int_{-d}^0 p(s-\theta) 
\bar f_x\left(s-\theta, \bar x_{s-\theta},\bar u_{s-\theta}\right)\mu_1(d\theta)\, ds\\
\qquad+\dis\int_t^T
\E^{\Fcal_s}\dis\int_{-d}^0q(s-\theta)\bar g_x\left(s-\theta,\bar  x_{s-\theta},\bar u_{s-\theta}\right)
\mu_2(d\theta)\, ds \\
\qquad
+\dis\int_t^Tq(s)dW_s+\dis\int_{t\vee(T-d)}^T\textcolor{blue}{\E^{\Fcal_s}}\bar  h_x(\bar x_T)\textcolor{blue}{\mu^T(ds)}
\\
 \; p(T-\theta)=0, q(T-\theta) =0 \quad \forall \,\theta \in [-d,0 ).
 \end{array}
 \right.
 \end{equation}
Let $(\bar x, \bar u)$ be an optimal pair for the optimal control problem of minimizing the cost functional (\ref{costo:fin}) related
to the controlled state equation (\ref{eq:state:fin}).
Then the following condition holds: 
\begin{align}\label{max-princ:fin:condiz-ham}
 & \< v(t)-\bar{u}(t),\E^{\mathcal{F}_t}\int_{-d}^0\bar{f}_u(t-\theta,\bar x_{t-\theta},\bar u_{t-\theta})p(t-\theta)\mu_3(d\theta) \> \nonumber\\  &+ 
    \< v(t)-\bar{u}(t),\E^{\mathcal{F}_t} \int_{-d}^0\bar{g}_u(t-\theta,\bar x_{t-\theta},\bar u_{t-\theta})q(t-\theta)\mu_4(d\theta) \> 
  \\ & +   \< v(t)-\bar{u}(t), \E^{\mathcal{F}_t}\int_{-d}^0\bar  l_u\left(t-\theta, \bar x_{t-\theta},\bar u_{t-\theta}\right) \mu_6(d\theta)  \>   \geq 0  \qquad\quad  dt\times \P- a.e.; \nonumber
\end{align}
for all $ v \in \mathcal{U}$.
\end{theorem}
\begin{remark}
We notice that in  equation \eqref{ABSDEtrascinata-aggiunta-mu} and in condition \eqref{max-princ:fin:condiz-ham}  the terms  
$\E^{\mathcal{F}_t}\int_{-d}^0\bar{g}_x(t-\theta,x_{t-\theta}, u_{t-\theta})q(t-\theta)\mu_2(d\theta)$  and
$\E^{\mathcal{F}_t}\int_{-d}^0\bar{g}_u(t-\theta,x_{t-\theta}, u_{t-\theta})q(t-\theta)\mu_4(d\theta)$ 
make sense only when integrated iwith respect to $t$  as we already pointed out for the control terms.  
\newline As it is well known, the stochastic maximum principle can be reformulated without differentiability assumptions on the coefficients as stated in Hypothesis \ref{ip:findim}. In the place of differentiability, we assume that $\bar f$ and $\bar g$ are Lipschitz continuous with respect to $x,\,u$, $\bar l$ is locally Lipschitz continuous with respect to $x,\,u$, and $\bar l$ is locally Lipschitz continuous.
In this case condition (\ref{max-princ:fin:condiz-ham-var}) can be replaced by a condition on the variation of the Hamiltonian function. 
 Namely let  $v$ be another admissible control, set $\bar v= v-\bar u$ and $u^\rho= \bar u+\rho \bar v$, condition (\ref{max-princ:fin:condiz-ham-var}) can be substituted by
\begin{multline}\label{max-princ:fin:condiz-ham-var}
 \E^{\mathcal{F}_t}\int_{-d}^0\left(\bar{f}(t-\theta, \bar x_{t-\theta}, u^\rho_{t-\theta})-\bar{f}(t-\theta,\bar x_{t-\theta},\bar u_{t-\theta})\right)p(t-\theta)\mu_3(d\theta) \\
  + \E^{\mathcal{F}_t} \int_{-d}^0\left(\bar{g}(t-\theta, \bar x_{t-\theta}, u^\rho_{t-\theta})-\bar{g}(t-\theta,\bar x_{t-\theta},\bar u_{t-\theta})\right)q(t-\theta)\mu_4(d\theta)  \\+ \E^{\mathcal{F}_t}\int_{-d}^0\left(\bar l(t-\theta,  \bar x_{t-\theta}, u^\rho_{t-\theta})-  \bar l(t-\theta,  \bar x_{t-\theta}, \bar u_{t-\theta})\right)\mu_6(d\theta)\geq 0  ,
\end{multline}
$ dt\times d\P-$ a.e.. This form of the maximum principle can be obtained in a similar to the differentiable case, without writing the variation of the coefficients in terms of derivatives.
\newline Finally we notice that, unlike the undelayed case, both conditions \eqref{max-princ:fin:condiz-ham} and \eqref{max-princ:fin:condiz-ham-var}
 cannot be expressed with any derivative or variation of the Hamiltonian.
\end{remark}
\textbf{Proof of Theorem \ref{maxprinc-findim}.} 
As we already pointed out, see the comment after the proof of Theorem \ref{proposition:ABSDEtrascinata}, the adjoint equation  \eqref {ABSDEtrascinata-aggiunta-mu}  is not regular enough to perform directly the usual proof of the maximum principle. Thus during the proof we must introduce some suitable regularized  approximating problem to apply the It\^o formula and  deduce the necessary condition \eqref{max-princ:fin:condiz-ham}.
\newline As usual in proving the stochastic maximum principle, we start by writing the variation of the cost functional. Namely, following (\ref{urho}), let $(\bar x,\bar u)$ be an optimal pair
 and let $v$ be another admissible control, set $\bar v= v-\bar u$
and $u^\rho= \bar u+\rho \bar v$.  
We can write the variation of the cost functional, $$\delta J = J(u^\rho(\cdot))-J(\bar u(\cdot)),$$ as 
\begin{align}\label{costo:fin:variation}
 0&\leq \delta J = J(u^\rho(\cdot))-J(\bar u(\cdot))\\ \nonumber 
 & =  \E \int_0^T    l(t, x^\rho_t, u^\rho_t)dt-
  \E \int_0^T  l(t,\bar x_t,\bar u_t)dt  + \E \left(h(x^\rho_T)- h(\bar x_T)\right)=I_1+I_2. \nonumber
 \end{align}
 Now
 \begin{align*}
I_1 &=\E \int_0^T l(t, x^\rho_t, u^\rho_t)\,dt-
  \E \int_0^T  l(t,\bar x_t,\bar u_t)\,dt\\
  & =\left[\E \int_0^T    l(t, x^\rho_t, u^\rho_t)\,dt-
  \E \int_0^T  l(t,\bar x_t, u^\rho_t)\,dt \right] \\ 
  &+\left[\E \int_0^T    l(t, \bar x_t, u^\rho_t)\,dt-
  \E \int_0^T  l(t,\bar x_t,\bar u_t)\,dt\right]=J_1+J_2 
 \end{align*} 
 We rewrite (\ref{resto:fin:rho}) as
\begin{equation}\label{resto:fin:rho-bar}
x^\rho(t)=\bar x(t)+\rho y(t)+R^\rho(t),\, t\in[0,T],\quad \lim_{\rho\to 0}\frac{1}{\rho^2}\E\sup_{t\in[0,T]}\vert R^\rho(t)\vert^2=0,
 \end{equation}
where $y$ is solution to equation (\ref{eq:state:fin:variation}). We start by computing $J_1$:
  \begin{align}\label{J1}
&J_1=\E \int_0^T   \left[ l\left(t, \int_{-d}^0x^\rho(t+\theta)\mu_5(d\theta) ,\int_{-d}^0u^\rho(t+\theta)\mu_6(d\theta)\right)\right.\\  \nonumber  
&\left.-l\left(t, \int_{-d}^0\bar x(t+\theta)\mu_5(d\theta) ,\int_{-d}^0u^\rho(t+\theta)\mu_6(d\theta)\right)dt \right]\\  \nonumber  
&=  \E \int_0^T \int_0^1 \int_{-d}^0\left(\textcolor{blue}{x^\rho(t+\theta)- \bar x(t+\theta)}\right)\bar{l}_x\left(t,\left(\bar x_t+\lambda\left( { x^\rho_t-\bar x_t}\right)\right), u^\rho_t)\right)\mu_5(d\theta) d\lambda dt \\ \nonumber  
&=\E \int_0^T \int_0^1\int_{-d}^0\left(\rho y_t(\theta) +R^\rho(t+\theta)  \right)  \bar l_x\left(t,\bar x_t+\lambda\left( x^\rho_t-\bar x_t\right), u^\rho_t\right)\mu_5( d\theta)d\lambda dt.  \\ \nonumber
\end{align}
By similar computations we obtain the analogous formula for $J_2$:
 \begin{align}\label{J2}
J_2& =  \E \int_0^T \int_0^1\int_{-d}^0 \left(u^\rho\textcolor{blue}{(t+\theta)}-\bar u\textcolor{blue}{(t+\theta)}  \right)  
 \bar{l}_u\left(t,\bar x_t, u^\rho_t+\lambda\left( u^\rho_t-\bar u _t\right)\right) \mu_6(d\theta)d\lambda dt \\ \nonumber
  & = \E \int_0^T \int_0^1\int_{-d}^0  \rho \bar v(t+\theta) \bar{l}_u\left(t,\bar x_t, u^\rho_t+\lambda\left( u^\rho_t-\bar u _t\right)\right)  \mu_6(d\theta)d\lambda dt .\\ \nonumber  
\end{align}
Notice that the last term is well defined only when $\bar u, u^\rho$ are continuous, i.e. belong to $E$, but can be extended to the whole $L^2([-d,T];U)$ by a standard density argument. 
We now compute $I_2$:
\begin{align}\label{I2}
 I_2&=\E \left(h(x^\rho_T)- h(\bar x_T)\right)\\ \nonumber
  &= \E \int_0^1\int_{-d}^0 \left(x^\rho(T+\theta) -\bar x(T+\theta))\right)
  \bar{ h} _x\left(\bar x_T+\lambda\left(x^\rho_T-\bar x_T \right)\right) \mu (d\theta) \\ \nonumber
  &= \E \int_0^1\int_{-d}^0\left(\rho y(T+\theta) +R^\rho(T+\theta)\right)
 \bar{  h} _x\left(\bar x_T+\lambda\left(x^\rho_T-\bar x_T \right)\right) \mu (d\theta)\\ \nonumber
& = \E \int_0^1\int_{-d}^0\left(\rho y(T+\theta) +R^\rho(T+\theta)\right) \bar{ h}_x\left(\bar x_T+\lambda\left(x^\rho_T-\bar x_T \right)\right)
   \mu(d\theta). \nonumber
  \end{align}
 Now we follow Lemma \ref{lemma:approx-measure} and we decompose the measure $\mu$ into
 \begin{equation}\label{dec:mu_h}
 \mu= \bar \mu + \mu(\left\lbrace 0\right\rbrace)\delta_0
 \end{equation}
so that $\bar\mu$ turns out to be a finite regular measure on $[-d,0]$, such that $\bar \mu (\left\lbrace 0 \right\rbrace)=0$.
 By Lemma \ref{lemma:approx-measure} there exists a sequence $(\bar\mu^{n})_{n\geq 1}$ of finite regular measures on $[-d,0]$, absolutely continuous with respect to
 $\lambda_{[-d,0]}$, the Lebesgue measure on $[-d,0]$, such that
 \begin{equation}\label{approx:barmu}
  \bar\mu=\lim_{n\rightarrow\infty}\bar \mu^{n}.
 \end{equation}
 So following (\ref{I2}), the variation of the final cost can be written as
 
 \begin{align}\label{I2:approx}
 I_2&= \lim_{n\rightarrow \infty} \E \int_0^1\int_{-d}^0 \left(\rho y_T(\theta) +R^\rho(T+\theta)\right) \bar h_x\left(\bar x_T+\lambda\left(x^\rho_T-\bar x_T \right)\right)
  d\bar\mu_{n}(\theta)d\lambda  \\ \nonumber
  &+\E \int_0^1\left(\rho y_T(\theta) +R^\rho(T+\theta)\right) \bar h_x\left(\bar x_T+\lambda\left(x^\rho_T-\bar x_T \right)\right)
  \mu(\left\lbrace 0\right\rbrace)d\lambda 
  \end{align}
  So taking into account the computation for $J_1$, $J_2$ and for $I_2$ that we have performed in (\ref{J1}), (\ref{J2}), (\ref{I2})
  and (\ref{I2:approx}), also by dividing both sides of (\ref{costo:fin:variation}) by $\rho$, and then by
  letting $\rho\to 0$ on the right hand side, we get
 \begin{align}\label{costo:fin:variation:approx}
 0&\leq  \E \int_0^T \int_{-d}^0 \textcolor{blue}{y(t+\theta)}\bar  l_x(t, \bar x_t,\bar u_t)\mu_5(d\theta)dt +
 \E \int_0^T   \int_{-d}^0 \bar{l}_u(t, \bar x_t,\bar u_t)\textcolor{blue}{\bar v(t+\theta)}  \mu_6(d\theta)dt \\ \nonumber
&+ \mu_h(\left\lbrace 0\right\rbrace)\bar h_x(\bar x_T)y_T(0)
+ \lim_{n\rightarrow \infty}\E \int_{-d}^0\textcolor{blue}{y(T+\theta)} \bar h_x(\bar x_T)d\bar\mu_{n}(\theta).
 \end{align}
Let $J^n$ the cost obtained from $J$ defined in \eqref{costo:fin} by replacing $\bar\mu $ with its absolute continuous approximation $\bar\mu_n$ in the final cost. So the variation $\delta J^n$ of $J^n$ is given by
 \begin{align}\label{costo:fin:variation:approx-delta}
  \delta J^n&  : =  J^n(u^\rho(\cdot))-J^n(\bar u(\cdot))=  \E \int_0^T \int_{-d}^0  \textcolor{blue}{y_t(\theta)\bar  l_x(t, \bar x_t,\bar u_t})\mu_5(d\theta)dt \\ \nonumber
&+\E \int_0^T  \bar{l}_u(t, \bar x_t,\bar u_t)\bar v_t(\theta) \mu_{6}(d\theta)dt \\ \nonumber
&+ \mu(\left\lbrace 0\right\rbrace)\bar h_x(\bar x_T)y_T(0)
+\E \textcolor{blue}{ \int_{T-d}^Ty(\theta) \bar h_x(\bar x_T)d\bar\mu^{n,T}}(\theta).
 \end{align}
 We notice that in the first term, $y_t(\theta)=y(t+\theta)=0$ if $t+\theta <0$, and the same holds for $\bar{v}$.
 
Since (\ref{ABSDEtrascinata-aggiunta-mu}) does not make sense in differential form and we cannot apply the Ito formula, we now introduce an approximated version of equation (\ref{ABSDEtrascinata-aggiunta-mu}), along the lines we have described at the beginning of the proof.
\newline First we notice that
with $\mu$ decomposed into $\bar\mu$ and $\mu(0)$ as in (\ref{dec:mu_h}), equation (\ref{ABSDEtrascinata-aggiunta-mu}) can be rewritten as
 \begin{equation}\label{ABSDEtrascinata-aggiunta-bar}
  \left\lbrace\begin{array}{l}
p(t)=\dis\int_t^T \E^{\Fcal_s}\dis\int_{-d}^0 {\bar l_x\left(s-\theta, \bar x_{s-\theta},\bar u_{s-\theta }\right)}
\mu_5(d\theta)\,ds\\
\qquad+\dis\int_t^T\E^{\Fcal_s}\dis\int_{-d}^0 p(s-\theta) 
\bar f_x\left(s-\bar \theta, x_{s-\theta},\bar u_{s-\theta}\right)\mu_1(d\theta)\, ds\\
\qquad+\dis\int_t^T
\E^{\Fcal_s}\dis\int_{-d}^0q(s-\theta) \bar g_x\left(s-\theta, \bar x_{s-\theta},\bar u_{s-\theta}\right)
\mu_2(d\theta)\, ds \\
\qquad+\dis\int_t^Tq(s)dW_s+\dis\int_{t\vee(T-d)}^T\textcolor{blue}{\E^{\Fcal_s}}\bar  h_x(\bar x_T)\textcolor{blue}{\bar\mu^T(ds)}
+\mu(\left\lbrace 0\right\rbrace)\bar h_x(\bar x_T)\\
p(T-\theta), \; q(T-\theta)=0 \quad \forall \,\theta \in [-d,0 ).
 \end{array}
 \right.
 \end{equation}
Now we approximate $\bar\mu$  by $\bar\mu^{n}$ as in (\ref{approx:barmu}) in the ABSDE (\ref{ABSDEtrascinata-aggiunta-bar}), and so we obtain an approximated version of (\ref{ABSDEtrascinata-aggiunta-bar}) given by 
\begin{equation}\label{ABSDEtrascinata-aggiunta-approx-mu}
 \left\lbrace\begin{array}{l}
p^n(t)=\dis\int_t^T \E^{\Fcal_s}\dis\int_{-d}^0 
 l_x\left(s-\theta,\bar  x_{s-\theta},\bar u_{s-\theta}\right)
\mu_5(d\theta)\,ds\\
\qquad+\dis\int_t^T\E^{\Fcal_s}\dis\int_{-d}^0 p^n(s-\theta) 
\bar f_x\left(s-\theta, \bar x_{s-\theta},\bar u_{s-\theta}\right)\mu_1(d\theta)\, ds\\
\qquad+\dis\int_t^T
\E^{\Fcal_s}\dis\int_{-d}^0q ^n(s-\theta) \bar g_x\left(s-\theta,\bar  x_{s-\theta},\bar u_{s-\theta}\right)
\mu_2(d\theta)\, ds \\
\qquad
+\dis\int_t^Tq ^n(s)dW_s+\dis\int_{t\vee(T-d)}^T\textcolor{blue}{\E^{\Fcal_s}}\bar  h_x(\bar x_T)\textcolor{blue}{\bar\mu^{n,T}(ds)}
+\mu(\left\lbrace 0\right\rbrace)\bar h_x(\bar x_T)\\
p^n(T-\theta), \; q^n (T-\theta)=0 \quad \forall \,\theta \in [-d,0 ).
 \end{array}
 \right.
  \end{equation}
Since the differential form of $p^n(t)$ makes sense, we can compute $d\< y(t), p^n(t)\>$:
\begin{align*}
 &d\< y(t), p^n(t)\>\\ \nonumber
 &=\<d y(t), p^n(t)\>+\< y(t),d p^n(t)\>+\<\dis\int_{-d}^0 \bar g_x(t, \bar x_t,\bar u_t)\mu_2(d\theta),q^n(t)\>dt  \\ \nonumber&+\<\dis\int_{-d}^0 \bar g_u(t, \bar x_t,\bar u_t)\bar v(t+\theta)\mu_4(d\theta),q^n(t)\>dt\\ \nonumber
 &=\<\dis\int_{-d}^0 y(t+\theta)\bar  f_x(t, \bar x_t,\bar u_t)d\mu_1(\theta)\,dt + \int_{-d}^0  f_u(t,\bar x_t,\bar u_t)\bar v(t+\theta) \mu_6(d\theta)\,dt\\ \nonumber  &+ 
 \dis\int_{-d}^0y(t+\theta) \bar g_x(t,\bar x_t,\bar u_t)d\mu_2(\theta)\,dW(t)\\ \nonumber&+\dis\int_{-d}^0 \bar g_u(t, \bar x_t,\bar u_t)\bar v(t+\theta) \mu_{4}(d\theta)\,dW(t),p^n(t)\>\\ \nonumber
 &-\<y(t),  \E^{\Fcal_t}\dis\int_{-d}^0
\bar l_x\left(t-\theta, \bar x_{t-\theta},\bar u(t-\theta)\right)
\mu_5(d\theta)\>dt\\ \nonumber
&-\<y(t),\E^{\Fcal_t}\dis\int_{-d}^0 p^n(t-\theta)
\bar f_x\left(t-\theta, \bar x_{t-\theta},\bar u_{t-\theta}\right)\mu_1(d\theta)\>dt\\ \nonumber
&-\<y(t),
\E^{\Fcal_t}\dis\int_{-d}^0q^n(t-\theta)\bar g_x\left(t-\theta, \bar x_{t-\theta}, \bar u_{t-\theta})\right)
\mu_2(d\theta) \>dt\\ \nonumber 
&-\<y(t),q(t)dW_t\>+\<y(t),\chi_{t>T-d}\bar  h_x(\bar x_T)\frac{d\textcolor{blue}{\bar\mu^{n,T}}}{dt}dt\>\\ \nonumber 
&-\<\dis\int_{-d}^0 \bar g_x(t,\bar  x_t, \bar u_t)\mu_2(d\theta),q^n(t)\>dt.
\end{align*}
Integrating between $0$ and $T$ and taking expectation we obtain
\begin{align*}
&\E\< y(T), \mu_h(\left\lbrace 0\right\rbrace)\bar h_x(x_T)\>
 =\E\dis\int_0^T\<\dis\int_{-d}^0y(t+\theta) \bar  f_x(t, x_t,u_t)d\mu_1(\theta)  \\ &+  \int_{-d}^0f_u(t,x_t,u_t)\bar v(t+\theta) \mu_{6}(d\theta),p^n(t)\>dt\\ \nonumber
 &-\E\dis\int_0^T\<y(t),\E^{\Fcal_t}\dis\int_{-d}^0 \chi_{t-\theta<T}
\bar l_x\left(t-\theta, x_{t-\theta},u_{t-\theta}\right)
\mu_5(d\theta)\>dt\\ \nonumber
&-\E\dis\int_0^T\<y(t),\E^{\Fcal_t}\dis\int_{-d}^0 p^n(t-\theta) \chi_{t-\theta<T}
\bar f_x\left(t-\theta, x_{t-\theta},u_{t-\theta}\right)\mu_1(d\theta)\>dt\\ \nonumber
&-\E\dis\int_0^T\<y(t),
\E^{\Fcal_t}\dis\int_{-d}^0q^n(t-\theta) \chi_{t-\theta<T}\bar g_x\left(t-\theta, x_{t-\theta},u_{t-\theta}\right)
\mu_2(d\theta)\>dt\\ \nonumber 
&-\E\dis\int_0^T \textcolor{blue}{\E^{\Fcal_t}}\<y(t),\chi_{t>T-d}\bar  h_x(x_T)\textcolor{blue}{\tilde{\bar\mu}^{n,T}}(t)\>dt\\ \nonumber 
&-\E\dis\int_0^T\<\dis\int_{-d}^0 \bar g_x(t, x_t,u_t)d\mu_2(\theta),q^n(t)\>dt
\end{align*}
 where 
 \[
 \textcolor{blue}{ \tilde{\bar\mu}^{n,T}}=\frac{\textcolor{blue}{d\bar\mu^{n,T}}}{dt},
   \]
is the Radon Nikodym derivative of $\bar\mu^{n,T}$ with respect to the Lebesgue measure.
By some change in the time variable and with the optimal pair $(\bar x, \bar u)$ instead of $(x,u)$, it turns out that
\begin{align*}
\delta J^n&=\E\< y(T), \mu(\left\lbrace 0\right\rbrace)\bar h_x(\bar x_T)\>+
\E\dis\int_0^T\textcolor{blue}{\E^{\Fcal_t}}\<y(t),\chi_{t>T-d}\bar  h_x(\bar x_T) \textcolor{blue}{\tilde{ \bar\mu}^{n}(t)}\>dt\\ \nonumber
&+\E\dis\int_0^T\<y(t),\E^{\Fcal_t}\dis\int_{-d}^0 \chi_{t-\theta<T}
\bar l_x\left(t-\theta, \bar x_{t-\theta},\bar u_{t-\theta}\right)
\mu_5(d\theta)\>dt\\ \nonumber
&+\E \int_0^T  l_u(t, \bar x_t,\bar u_t) \int_{-d}^0 \bar v(t+\theta) \mu_6(d\theta)dt\\ \nonumber
 &=\E\dis\int_0^T\<\int_{-d}^0\bar{f}_u(t,\bar x_t,\bar u_t)\bar v(t+\theta) \mu_3(d\theta),p^n(t)\>dt\\ \nonumber
&+ \E\dis\int_0^T\<\int_{-d}^0\bar{g}_u(t,\bar x_t,\bar u_t)\bar v(t+\theta) \mu_4(d\theta),q^n(t)\>dt \\ &+
 \E \int_0^T  \int_{-d}^0 l_u(t, \bar x_t,\bar u_t)  \bar v(t+\theta) \mu_6(d\theta)dt. \nonumber
 \end{align*}
So, taking into account (\ref{costo:fin:variation:approx})
\begin{align*}
0&\leq\E\dis\int_0^T\<\int_{-d}^0\bar{f}_u(t,\bar x_t,\bar u_t)\bar v(t+\theta) \mu_3(d\theta),p^n(t)\>dt\\ \nonumber
& + \E\dis\int_0^T\<\int_{-d}^0\bar{g}_u(t,\bar x_t,\bar u_t)\bar v(t+\theta) \mu_4(d\theta),q^n(t)\>dt\\ & +\E \int_0^T  \int_{-d}^0  l_u(t, \bar x_t,\bar u_t)\bar v(t+\theta) \mu_6(d\theta)dt
\end{align*}
and letting $n\rightarrow \infty$ we get
\begin{align}\label{max-princ:fin:condiz-ham-diffle}
&0\leq\E\dis\int_0^T\<\int_{-d}^0\bar{f}_u(t,\bar x_t,\bar u_t)\bar v_t(\theta) \mu_3(d\theta),p(t)\>dt \\ \nonumber
&+ \E\dis\int_0^T\<\int_{-d}^0\bar{g}_u(t,\bar x_t,\bar u_t)\bar v_t(\theta) \mu_4(d\theta),q(t)\>dt\\ \nonumber&\quad +\E \int_0^T  \int_{-d}^0 l_u(t, \bar x_t,\bar u_t)\bar v(t+\theta )\mu_6(d\theta) dt
\\ \nonumber& =  \E\dis\int_0^T\< \bar v(t) , \E^{\mathcal{F}_t}\int_{-d}^0\bar{f}_u(t-\theta,\bar x_{t-\theta},\bar u_{t-\theta})p(t-\theta)\mu_3(d\theta) \>dt \\ \nonumber
&\quad+ \E\dis\int_0^T\< \bar v(t ), \E^{\mathcal{F}_t} \int_{-d}^0\bar{g}_u(t-\theta,\bar x_{t-\theta},\bar u_{t-\theta})q(t-\theta)\mu_4(d\theta) \>dt\\ \nonumber
&\quad +\E \int_0^T \< \bar v(t ), \E^{\mathcal{F}_t} \int_{-d}^0  l_u(t-\theta, \bar x_{t-\theta},\bar u_{t-\theta})\bar v(t) \mu_6(d\theta) \>dt 
\end{align}
which is nothing else than (\ref{max-princ:fin:condiz-ham}) in integral form. The conclusion follows by a standard localization procedure, along the lines given e.g. in \cite{OrrRocSca2020}, end of paragraph 5.4, see also \cite{hu1996maximum}, end of the proof of Theorem 5.1, and \cite{Yo99}. 
\qed

\section{Delay equations arising in advertising models}
\label{sec:opt-adv}
We consider a stochastic dynamic model in marketing for problems of optimal advertising. We study, as done in \cite{GozMar} and in \cite{GozMarSav}, stochastic models for optimal advertising starting from the stochastic variant introduced in \cite{GrossVisc}, and also with delay both in the state and in the control, see also \cite{Hartl}. In this model delay in the control corresponds to lags in the effect of advertisement.
\newline So we consider, for $t\in[0,T]$, the following controlled stochastic differential equation in $\R$ with delay in the state and in the  control:
\begin{equation}
\left\{
\begin{array}
[c]{l}%
dy(t)  =\left[a_0 y(t) +\displaystyle\int_{-d}^0y(t+\theta)\mu_a(d\theta)+b_0 u(t) +\displaystyle\int_{-d}^0u(t+\theta)\mu_b(d\theta)\right]\,dt\\ \qquad\qquad +\sigma_a y(t)dW_t+\sigma_b u(t)dW_t
, \\
y(\theta)  =y_0(\theta), \quad \theta \in [-d,0),\\
u(\theta)=u_0(\theta), \quad \theta \in [-d,0).
\end{array}
\right.  \label{eq-contr-rit}
\end{equation}
In equation \eqref{eq-contr-rit}, $y$ represents the goodwill level, $a_0$ is a constant factor of image deterioration in absence of advertising, $b_0$ is a constant representing an advertising effectiveness factor, $\mu_a(\cdot)$ is the distribution of the forgetting time, and $\mu_b(\cdot)$ is the distribution of the time lag between the advertising expenditure $u$ and the corresponding effect on the goodwill level. The diffusion term $\sigma_a y(t)$ accounts for the word of mouth communication, the parameter $\sigma _a$ is the advertising volatility; the diffusion term, $\sigma_b u(t)$ accounts for the effect of advertising, the parameter $\sigma_b$ is the communication effectiveness volatility. Moreover, $y_0(0)$ is the level of goodwill at the beginning of the advertising campaign, while $y_0(\cdot)$ is the history  of the goodwill level before the initial time, and $u_0(\cdot)$ is the history of the advertising expenditure before the initial time, too.
\newline We assume the following:
\begin{hypothesis}\label{ipotesibasic}
\begin{itemize}
\item[(i)] $W$ is a standard Brownian motion in $\R$, and $(\mathcal{F}_t)_{t\geq 0}$ is the augmented filtration generated by $W$;
\item[(ii)] $a_0, \,\sigma_a,\, \sigma_b\in \R$;
\item[(iii)] the control strategy $u$ belongs to $\mathcal{U}$ where $$\mathcal{U}:=\left\lbrace z\in \mathcal L^2_{\mathcal{F}}(\Omega\times [0,T], \R):u(t)\in U \;a.s.\right\rbrace $$ where $U$ is a convex subset of $\R$;
\item[(iv)] $d>0$ is the maximum delay the control takes to affect the system;
 \item[(v)] $\mu_a,\, \mu_b$ are finite regular measures in $[-d,0]$ that describe the time that respectively the state and the control take to affect the system.
\end{itemize}
\end{hypothesis}
The objective is to minimize, over all controls in $\mathcal {U}$, the following finite horizon cost:
 \begin{equation}\label{costo-advertisement}
J(t,x,u)=\E \int_t^T \ell\left(s,y(s),u(s)\right)\;ds +\E  \phi(y_T),
\end{equation}
where $\ell$ represents the cost of advertisement, and $-\phi$ represents the final utility, that may depend on the trajectory $y_T=y(T+\theta),\, \theta \in [-d,0].$
We assume that $\ell:[0,T]\times\R\times \R\rightarrow \R$ is continuous, bounded and differentiable with respect to $x$ and $u$, moreover the derivatives with respect to $x$ and $u$ satisfy
\[\abs{\ell_x(t,y,u)}+  \abs{\ell_ul(t,y,u)} \leq C_3(1 + \abs{y} + \abs{u}),  \]
and $\phi$ is given by
\begin{equation}\label{fi}
\phi(y_T)= \bar\phi\left(\int_{-d}^0 y( T+\theta) \mu_\phi(d\theta)\right),
\end{equation}
where $\bar\phi:\R\rightarrow \R$  is Lipschitz continuous and differentiable and $\mu_\phi$ is another finite regular measure on $[-d,0]$.
\newline We consider the adjoint equation for the pair of processes $(p,q)\in \Lcal^2_{\mathcal{F}}(\Omega\times [0,T], \R)\times 
\Lcal^2_{\mathcal{F}}(\Omega\times [0,T], \R)$
 \begin{equation}\label{aggiunta-advertisement}
  \left\lbrace\begin{array}{l}
p(t)= \dis\int_t^T \ell_x\left(s, y(s),u(s)\right)\, ds
+ \dis\int_t^Ta_0p(s)\,ds\\ \qquad+\dis\int_t^T\E^{\Fcal_s}\dis\int_{-d}^0 p(s-\theta)\mu_a(d\theta)\, ds+ \dis\int_t^T\sigma_aq(s)\,ds+\dis\int_t^Tq_sdW_s\\ \qquad+\dis\int_{t\vee(T-d)}^T\textcolor{blue}{\E^{\Fcal_t}}\bar  \phi_x(y_T) \textcolor{blue}{ \mu_\phi^T}(d\theta)\\
p(T-\theta)=0, \; q(T-\theta)=0 \quad \forall \,\theta \in [-d,0 ).
 \end{array}
 \right.
 \end{equation}
\begin{theorem}\label{maxprinc-findim-advertisement}
 Let Hypothesis \ref{ipotesibasic} hold true. Let $(p,q)$ be the unique solution of the ABSDE (\ref{aggiunta-advertisement}). 
Let $(\bar y, \bar u)$ be an optimal pair for the optimal control problem of minimizing the cost functional (\ref{costo-advertisement}) related to the controlled state equation (\ref{eq-contr-rit}). Let  $v$ be another admissible control, set $\bar v= v-\bar u$ and $u^\rho= \bar u+\rho \bar v$, then 
\begin{multline}\label{max-princ:fin:condiz-ham-advertisement}
b_0\left( \bar u(t)- u^\rho(t)\right)p(t)+ \left( \bar u(t)- u^\rho(t)\right)\E^{\mathcal{F}_t} \int_{-d}^0p(t-\theta)\mu_b(d\theta)\\
+\sigma_b\left( \bar u(t)- u^\rho(t)\right)q(t)+\ell_u(t,\bar y(t),\bar u(t))\left( \bar u(t)- u^\rho(t)\right)\leq 0  \quad dt\times \P\; a.s..
\end{multline}
\end{theorem}

\section{An optimal portfolio problem with execution delay}
\label{sec:opt-port}

We consider a generalized Black and Scholes market with one risky asset, whose price at time $t$ is denoted by $S(t)$ and whose past trajectory from time $t-d$ up to time $t$ is denoted by $S_t$, and one non-risky asset, whose price at time $t$ is
denoted by $B(t)$. The result can be extended to the case of a Black and Scholes market with $j$ risky assets, whose prices at time t are denoted by $S^i(t),i = 1, . . . , j$, and one non-risky asset: for the sake  of simplicity we limit here to the case of only one risky asset.
\newline  The evolution of the prices is given by the following stochastic delay differential equation in a complete probability space $(\Omega, \Fcal, \mP)$ : 
\begin{equation}\label{eq:price_evolution}
\left\lbrace
\begin{array}{l}
dS(t)=S(t )\left[ b(t, S_t ) dt+\sigma(t, S_t ) dW_t  \right],\
S(\theta)=\nu_0(\theta),\vspace{0.1cm}\\
dB(t)=r(t, S_t )B(t) dt,\vspace{0.1cm}\\
B(0)=B_0
\end{array}
\right.
\end{equation}
where $W(t)$ is a standard Brownian motion in $\R$, $(\Fcal_t)_{t\geq 0}$ is the filtration generated by $W$ and augmented with null probability sets and $S_t(\theta)=S(t+\theta),\, \theta \in [-d,0]$. 
The drift $b$, the diffusion $\sigma$ and the rate $r$ are given by
\begin{align}\label{eq:b-r-sigma}
&b(t, S_t )=\bar b(t, \int_{-d}^0S(t+\theta )\mu_{\bar b}(d\theta)), 
\quad \sigma(t, S_t )=\bar \sigma(t, \int_{-d}^0S(t+\theta )\mu_{\bar \sigma}(d\theta)) ,\nonumber \\ &  r(t, S_t )=\bar r (t\int_{-d}^0S(t+\theta )\mu_{\bar r}(d\theta))
\end{align}
where $\mu_{\bar b},\,\mu_{\bar \sigma},\,\mu_{\bar r}$ are finite regular measures on $[-d,0]$.

\begin{hypothesis}\label{ip: b_r_sigma}
On $\bar b, \bar \sigma  $ and $\bar r$ we make the following assumptions: 
\begin{itemize}
\item[i)]   $\mu_{\bar b}$ is a regular measure and $\bar b : [0,T]\times \R \rightarrow \R$  is measurable. Moreover $\forall\, s_i\in \R,\,i=1,2$
\[
\vert \bar b(t, s_1) -\bar b(t, s_2)\vert\leq c \vert s_1-s_2\vert\]
for some $c>0$ and for all $t\in[0,T]$ $\bar b(t,\cdot)$ is differentiable;
\item[ii)] $\mu_{\bar \sigma}$ is a regular measure and $\bar\sigma : [0,T]\times \R\rightarrow \R$ is measurable. Moreover $\forall\, s_i\in \R,\,i=1,2$
\[
\vert \bar\sigma(t, s_1) -\bar\sigma(t, s_2)\vert\leq c \vert s_1- s_2\vert
\] for some $c>0$ and for all $t\in[0,T]$ $\bar \sigma(t,\cdot)$ is differentiable;
\item[iii)] $\bar r : [0,T]\times \R\rightarrow \R$ is measurable. Moreover $\forall\, s_i\in \R,\,i=1,2$
\[
\vert r(t, s_1) -r(t, s_2)\vert\leq c \vert s_1- s_2\vert
\]for some $c>0 $ and for all $t\in[0,T]$ $\bar r(t,\cdot)$ is differentiable.
\end{itemize}
\end{hypothesis}
We now consider the evolution of $V(t)$, the value at time $t$ of the associated self-financing portfolio. We consider an optimal portfolio problem with execution delay, which is inspired by the models studied, in a different context, in \cite{BruPh}  in a stochastic impulse control framework, and which is treated also in \cite{FabFed}. Let $d>0$ be a fixed execution delay time:
at time $t>0$ the investor chooses, on the basis of the information contained in $\mathcal F_t$, to allocate the amount
of money $u(t)>0$ of its portfolio in the risky asset. This is the control process. However, due to the execution delay this order will be executed at time $t+d$ when the price of the risky asset has changed, see \cite{BruPh} for the formulation of this problem in a stochastic impulse control framework. Moreover we allow consumption, and also the investors are allowed to take money from the portfolio $V$: in the model this is represented by a further control $c$.
\newline The state equation for the optimal portfolio with execution delay is similar to the one considered in \cite{ElKa1997} in the case without delay, see also \cite{FabFed}, and it is given by
\begin{equation}\label{eq:portfolio}
\left\lbrace\begin{array}{ll}dV(t)= r (t, S_t)(V(t)- \pi^*(t-d)) dt-c(t)dt+\pi^*(t-d)\left[ b(t,S_t ) dt\right.\\ \qquad\,\left.+\sigma(t,S_t ) dW_t )\right]
\\
V(\theta)=\eta(\theta),  \,  \pi(\theta)= \pi_0(\theta),\, \theta\in[-d,0).
\end{array}
\right.
\end{equation}
Here we will only consider square-integrable, predictable investment strategies $\pi\in \mathcal{L}^2_\mathcal{F}(\Omega\times [0, T ] , \R)$.

The aim is to maximize the utility functional over the set of the admissible strategies
\begin{equation}\label{utility}
U(c)={\mathbb E} \int_0^T\left[U_1\left(t,c(t)\right) \right]\,dt +{\mathbb E} \left[U_2\left(\int_{-d}^0\ V( T+\theta) \mu_U(d\theta)\right) \right]=  {\mathbb E} \left[U\left( V_T\right) \right],
\end{equation}
where $U_1:[0,T]\times \R\rightarrow \R$ and  $U_2:\R\rightarrow \R$ are given utility functions, $U_1$ represents the utility from consumption and it is assumed to be continuous, differentiable in the second variable and the derivative with respect to $c$ has linear growth in $c$, and $U_2$ represents the utility from the wealth on $[T-d, T]$ and it is assumed to be Lipschitz continuous and differentiable. Here $\mu_U$ is another finite regular measure on $[-d,0]$: the utility is related  not only to the final value $T$, but to the value of the portfolio in the window $[T-d,T]$, and so it depends on $V(T+\theta),\, \theta \in [-d,0]$.

At any time $t\in [-d,T]$, the state $X(t)\in \R^2 $ is given by the pair
\[
X(t)=\left(
\begin{array}{l}
S(t)\\
V(t)
\end{array}
\right).
\]
So the equation for $X$ is given by
\begin{equation}\label{eq:pair_X}
\left\lbrace
\begin{array}{l}
d\left(
\begin{array}{l}
S(t)\\
V(t)
\end{array}
\right)=\left(\begin{array}{l}S(t ) b(t,S_t )\vspace{2mm}\\
 r (t, S_t)(V(t)- \pi(t-d))-c(t)+\pi(t-d) b(t,  S_t )\end{array}\right) dt\vspace{3mm}\\\qquad\,
 +\left(\begin{array}{l}S(t)\sigma(t, S_t ) \vspace{2mm}\\
\pi(t-d) \sigma(t,  S_t )
\end{array}\right)dW_t \vspace{3mm}\\
\left(
\begin{array}{l}
S(\theta)\\
V(\theta)
\end{array}
\right)=\left(\begin{array}{l}
\nu_0(\theta)\\
\eta(\theta)
\end{array}
\right)
\end{array}
\right.
\end{equation}
and it turns out to be an equation with delay both in the state and in the control. 
\newline Notice that the adjoint processes are given by a pair of processes $$(p,q)=\left(\left(\begin{array}{l}p^1\\p^2\end{array}\right),\left(\begin{array}{l}q^1\\q^2\end{array}\right)\right)\in \Lcal^2_{\mathcal{F}}(\Omega\times [0,T], \R^2)\times 
\Lcal^2_{\mathcal{F}}(\Omega\times [0,T], \R^2)$$ solution of the ABSDEs we are going to write, and that it turns out that the pair $(p^1,q^1)$ is identically $0$. Indeed
 \begin{equation}\label{ABSDEtrascinata-aggiunta_execution delay1}
 \left\lbrace\begin{array}{l}
p^1(t)=\dis\int_t^Tp^1(s)b(s,S_s)\,ds+\dis\int_t^Tq^1(s)\sigma(s,S_s)\,ds+\dis\int_t^Tq^1(s)dW_s 
\\ \qquad+\dis\int_t^T\E^{\Fcal_s}\dis\int_{-d}^0 p^1(s-\theta) S(s-\theta) \bar b_x(s-\theta, S_{s-\theta})\mu_{\bar b}(d\theta)\, ds\\
\qquad+\dis\int_t^T\E^{\Fcal_s}\dis\int_{-d}^0 q^1(s-\theta) S(s-\theta) \bar \sigma_x(s-\theta, S_{s-\theta})\mu_{\bar \sigma}(d\theta)\, ds\\ 
\qquad\\
p(T-\theta)=0, \; q(T-\theta)=0 \quad \forall \,\theta \in [-d,0 ].
 \end{array}
 \right.
\end{equation}
The pair of processes $(p^2,q^2)\in \Lcal^2_{\mathcal{F}}(\Omega\times [0,T], \R)\times 
\Lcal^2_{\mathcal{F}}(\Omega\times [0,T], \R)$ satisfies the following equation:
 \begin{equation}\label{ABSDEtrascinata-aggiunta_execution delay2} \left\lbrace\begin{array}{l}
p^2(t)= \\ \dis\int_t^T\E^{\Fcal_s}\dis\int_{-d}^0 p^2(s-\theta) \left(V(s-\theta)-\pi(s-\theta-d)\right) \bar r_x(s-\theta, S_{s-\theta})\mu_{\bar r}(d\theta)\, ds\\
+\dis\int_t^T\E^{\Fcal_s}\dis\int_{-d}^0 p^2(s-\theta)\pi(s-\theta-d) \bar b_x(s-\theta, S_{s-\theta})\mu_{\bar b}(d\theta)\, ds\\
+\dis\int_t^T\E^{\Fcal_s}\dis\int_{-d}^0 q^2(s-\theta) \pi(s-\theta-d) \bar \sigma_x(s-\theta, S_{s-\theta})\mu_{\bar \sigma}(d\theta)\, ds\\
 +\dis\int_t^Tq^2(s)dW_s+\dis\int_{t\vee(T-d)}^T \textcolor{blue}{\E^{\Fcal_t}} U_x\left(V_ T\right)  \textcolor{blue}{ \mu_U^T}(d\theta)\\
p^2(T-\theta)=0, \; q^2(T-\theta)=0 \quad \forall \,\theta \in [-d,0 ).
 \end{array}
 \right.
\end{equation}
From the maximum principle stated in Theorem \ref{maxprinc-findim} we deduce the following condition on the optimal strategy for the present problem: notice that the optimality condition can be given only in terms of the pair of processes $(p^2,q^2)$.
\begin{theorem}\label{maxprinc-findim_executiondelay}
 Let Hypothesis \ref{ip:findim} holds true. Let $(p^2,q^2)$ be the unique solution of the ABSDE (\ref{ABSDEtrascinata-aggiunta_execution delay2})
Let $(\bar X, \bar \pi, \bar c)$ be an optimal pair for the optimal control problem of minimizing $-U$, where $U$ is defined in (\ref{utility}).
For every admissible control $(\pi_1,c_1)$  set $\bar \pi_1= \pi_1-\bar \pi,\, \bar c_1=c_1-\bar c $ and $\pi^\rho= \bar \pi+\rho \bar \pi_1,\,c^\rho= \bar c+\rho \bar c_1$, then  
\begin{multline}\label{max-princ:fin:condiz-ham-port}
\mathbb{E}^{\mathcal{F}_t}\left[\left(r (t+d, \bar S_{t+d})\left(\pi^\rho(t)- \bar\pi(t)\right) +c^\rho(t+d)-\bar c(t+d)\right.\right.\\ \nonumber
\left.\left.+\left(\bar \pi(t)- \pi^\rho(t)\right)  b(t+d,\bar S_{t+d} )\right)p^2(t+d)\right.\\ \nonumber\left.+\left(\bar\pi(t)-\pi^\rho(t)\right)\sigma(t+d,\bar S_{t+d}) q^2(t+d)\right] +\left(c^\rho(t)-\bar c(t)\right)(U_1)_c(t, \bar c(t)) \leq 0
\end{multline}
$dt\times \P\; a.s.,$ where
we have set $\bar X(t)=\left(
\begin{array}{l}
\bar S(t)\\
\bar V(t)
\end{array}
\right)$.
\end{theorem}

\bibliography{biblio}
 \bibliographystyle{plain}

 \end{document}